\documentclass[11pt, letterpaper, oneside]{article}

\usepackage{latexsym, amsmath, amssymb, amsfonts, amscd}
\usepackage{amsthm}
\usepackage{t1enc}
\usepackage[mathscr]{eucal}
\usepackage{indentfirst}
\usepackage{graphicx}   
\usepackage{fancyhdr}
\usepackage{fancybox}
\usepackage{enumerate,psfrag}
\usepackage[all,poly,web,knot]{xy}
\usepackage{epsfig} 
\usepackage{array}

\theoremstyle{plain}
\newtheorem{thm}{Theorem}[section]

\newtheorem{prop}[thm]{Proposition}
\newtheorem{lemma}[thm]{Lemma}
\newtheorem{cor}[thm]{Corollary}

\newtheorem{claim}[thm]{Claim}

\renewcommand{\latticebody}{\drop@{ }}

\theoremstyle{definition}
\newtheorem{defi}[thm]{Definition}

\theoremstyle{remark}
\newtheorem{remark}[thm]{Remark}
\newtheorem{ep}[thm]{Example}

\newcommand{\R}{\ensuremath{\mathbb R}}

\newcommand{\g}{\ensuremath{\frak{g}}}

\renewcommand{\d}{\text{d}}

\newcommand{\G}{\mathcal{G}}

\newcommand{\cK}{\mathcal{K}}
\newcommand{\cS}{\mathcal{S}}

\newcommand{\cG}{\mathcal{G}}
\newcommand{\cH}{\mathcal{H}}

\newcommand{\Lie}{\mathcal{L}}
\newcommand{\tto}{\rightrightarrows}						

\DeclareMathOperator{\id}{id} \DeclareMathOperator{\im}{Im}

\newcommand{\ta}{\tilde{a}}

\newcommand{\tg}{\tilde{\gamma}}

\newcommand{\talpha}{\tilde{\alpha}}

\newcommand{\K}{\mathcal K}                   
\newcommand{\Dd}{\mathcal{D}}

\def\U{{\mathcal U}}

\def\PB(#1,#2,#3,#4){
\left\{\begin{matrix}#1&\!\!\!\stackrel{?}{\longrightarrow}&\!\!\!#2\\
\downarrow&&\!\!\!\downarrow\\
#3&\!\!\!\stackrel{?}{\longrightarrow}&\!\!\!#4\end{matrix}\right\}}

\def\pb(#1,#2,#3,#4){ \hom(#1 \to #3, #2 \to #4)}

\def\codi(#1,#2,#3,#4){
\left\{\begin{matrix}#1&\!\!\!\longrightarrow&\!\!\!#2\\
\downarrow&&\!\!\!\downarrow\\
#3&\!\!\!\longrightarrow&\!\!\!#4\end{matrix}\right\}}

\newenvironment{Enumeratei}{
\begin{enumerate}[i)]
  \setlength{\itemsep}{1pt}
  \setlength{\parskip}{0pt}
  \setlength{\parsep}{0pt}}{\end{enumerate}}

\newenvironment{Itemize}{
\begin{itemize}
  \setlength{\itemsep}{2pt}
  \setlength{\parskip}{0pt}
  \setlength{\parsep}{0pt}}{\end{itemize}}

\newcommand {\comment}[1]{{\marginpar{*}\scriptsize{\bf Comments:}\scriptsize{\ #1 \ }}}

\begin{document}

\title{Lie algebroid Fibrations
\thanks
 {Supported by the German Research Foundation
(Deutsche Forschungsgemeinschaft (DFG)) through the Institutional
Strategy of the University of G\"ottingen, and the Funda\c{c}\~{a}o para a Ci\^{e}ncia e a Tecnologia (FCT / Portugal).
 }
} 
\author{Olivier Brahic  \\
Department of Mathematics, IST
\\\vspace{3mm}
email: brahic@math.ist.utl.pt\\
Chenchang Zhu\\
Courant Research Centre ``Higher Order Structures'',
University of G\"ottingen\\
email:zhu@uni-math.gwdg.de}
\date{}


\maketitle

\begin{abstract}
 A degree 1 non-negative graded super manifold equipped with a degree 1 vector field $Q$ satisfying $[Q, Q]=1$,
 namely a so-called NQ-1 manifold is, in plain differential geometry language, a Lie algebroid. 
 We introduce a notion of fibration for such super manifols, that essentially involves a complete Ehresmann connection. 
 As it is the case for Lie algebras, such fibrations turn out not to be just locally trivial products. 
We also define homotopy groups and prove the expected long exact sequence associated to a fibration.
In particular, Crainic and Fernandes's obstruction to the integrability of Lie algebroids is interpreted as the image
of a transgression map in this long exact sequence.
\end{abstract}

\tableofcontents

\section{Introduction}

There are many reasons to study the notion of fibration for Lie algebroids. For instance, one may want to understand
 better the topology behind the work \cite{cf} of Crainic and Fernandes on the obstruction to integrability of Lie algebroids. As we will see,
 these obstructions are the image by a transgression map that fits inside a long exact sequence of homotopy groups of Lie algebroids
 (see Cor. \ref{cor:A-L}).

On the other hand, it is well-known that the concept of Lie algebroid has deep mathematical physics background:
 a Lie algebroid is a degree 1 non-negative graded super manifold with a degree 1 vector field $Q$, satisfying of $[Q, Q]=1$.
 Such a super manifold is often called an NQ-1 manifold. Regardless the degree, NQ manifolds in general appear in BV
 theory and topological sigma models. In order to study NQ-bundles \cite{kotov-strobl}, a good definition of
 fibration is unavoidable.

A degree 1 non-negative graded super manifold can be modeled by a vector bundle with shifted degree $A[1]\to M$.  The function ring of $A[1]$ is the graded algebra 
\[ C(A) = C^\infty(M) \oplus \Gamma( A^*) \oplus \Gamma(\wedge^2 A^*) \oplus \cdots \]
A degree 1 vector field $Q$ is a degree 1 differential of this algebra. Equivalently, this means that we have a vector
bundle morphism (called the {\em anchor} later on) $\sharp: A \to TM$ and a Lie bracket $[\,\ ,\ ]$ on the space of the sections of $A$.
In fact for a homogeneous element $\xi\in C(A)$, \begin{multline}(Q \xi) (X_0,\dots,X_n) =  \sum_{i<j} (-1)^{i+j}\xi([X_i, X_j], X_0,\dots,\hat{X}_i,\dots,\hat{X}_j,\dots,X_n) 
 \\+ \sum_{i=0}^n (-1)^{i} (\sharp X_i)(\xi(X_0,\dots,\hat{X}_i,\dots,X_n)) \end{multline}
 generalizes the usual formula of the de Rham differential. The equation $[Q, Q]=0$ is then equivalent to the following conditions:
 \begin{eqnarray}\label{eq:rho-bracket}
                                       [X, fY]&=& f[X, Y]+ (\sharp X)(f)Y,\\
\label{eq:rho-bracket2}\oint_{X,Y,Z} [[X,Y],Z]&=&0, 
\end{eqnarray} for any $X, Y, Z\in \Gamma(A)$ and $f\in C^\infty(M)$. Here $\sharp$ denotes also the induced map $\Gamma(A) \to \Gamma(TM)$ on the level of sections.

A vector bundle $A\to M$ with $\sharp$ and $[\ \, ,\  ]$ as above, and satisfying \eqref{eq:rho-bracket} \eqref{eq:rho-bracket2}
 is called a {\em Lie algebroid}. For instance, when the base manifold $M$ is a point, $A$ is a Lie algebra.
 We refer the reader to \cite{cw} for more interesting examples and details on Lie algebroids.

In order to define a good notion of fibration in the category of Lie algebroids, one may start with a surjective morphism:
$$\xymatrix{A_E\ar@{->>}[r]^{\pi}\ar[d]&A_B\ar[d]\\
            E\ar[r]^{p} &B.}$$
 From this point of view there is a natural candidate, namely  $\K:=\ker \pi$, that should play the role of the fibered structure.
 Indeed, for any $x\in B$, the restriction $\K|_{E_x}$ is a well-defined Lie algebroid over the fiber $E_x:=p^{-1}(x)$,
 where $p:E\to B$ is the map underlying $\pi$ (there, we shall require $p$ to be a submersion in order to ensure
 that the fibers are smooth manifolds).

Thus $\K$ is  a family of Lie algebroids on the fibers of $p$. Locally, this means that one can choose coordinate charts $(x,y)=(x_1,\dots,x_p,y_1,\dots,y_q)$ on $E$, and a basis $\kappa_1,\dots,\kappa_r$ of sections of $\K$ such that:
\begin{equation}\label{eq:kernel:local:structure}
[\kappa_i,\kappa_j]=\sum_{l=1}^r c_{i,j}^{\, l} \kappa_l,\quad\quad
\sharp_\K(\kappa_i)=\sum_{j=1}^r s_{i,j} \partial y_j,
\end{equation}
where $c_{i,j}^{\, l}$ and $s_{i,j}$ are functions of $(x,y)$, and $\sharp_\K$ is the anchor of $\K$. Here, $(x_1,\dots,x_p)$ denote coordinates on $B$ while $(y_1,\dots,y_q)$ parametrize the fibers.

Now a first attempt to define a notion of fibration would be to require that for any Lie sub-algebroid $A_U\to A_B$,
 the fibre product $A_U \times_{A_B}A_E$ identifies with a direct product $A_U \times A_F$, for some Lie
subalgebroid $A_F\subset A_E$. In \cite{kotov-strobl} for instance, fibrations are required to be locally trivial. 
Yet, this requirement seems too restrictive: even in the case of Lie algebras, given a short exact sequence
\begin{equation} \label{eq:extension-lie-alg}
\mathfrak{k}\hookrightarrow \mathfrak{g}\twoheadrightarrow \mathfrak{h},
\end{equation}
in general $\mathfrak{g}$ will not split (as a Lie algebra) into a direct sum $\mathfrak{g}=\mathfrak{k}\oplus \mathfrak{h}$.
 Still, one always obtains by integration an exact sequence of groups $K\hookrightarrow G\twoheadrightarrow H$ so that $G$
 is a locally trivial fibration over $H$ with fiber-type $K$, regardless of $\g$ splitting into $\mathfrak{k}\oplus \mathfrak{h}$
 or not.

Another strong restriction coming from the above requirement, is that it would force all the fibers $\K|_{E_x}$ to be isomorphic; indeed,
 it would imply that one could choose the coordinates in \eqref{eq:kernel:local:structure} in such a way that $c_{i,j}^{\, l}$
 are independent of $x$. In general, we know \cite{Br} that given a connection, parallel transport allows only to compare two fibers
 that lie over a same $A_B$-orbit. In fact, even the diffeomorphism class of the fiber $E_x$ may vary from one point to another.

Now recall that in the case of manifolds, the notion of (locally trivial) fibration can be defined by requiring the existence
 of a complete Ehresmann connection.

An {\em Ehresmann connection} of $M\xrightarrow{\pi} N$
determines a lift of any vector field $X$ on $N$ to a vector field $\sigma(X)$ on $M$ satisfying $\pi_* (\sigma(X))=X$;
 it is called a horizontal lift.

An Ehresmann connection is said to be {\em complete} if  $\sigma(X)$ is a complete vector field whenever $X$ is.
 If $M\to N$ admits a complete Ehresmann connection, then one can always decompose $M$ locally as
 $M \stackrel{loc}{\sim} \mathcal{U} \times F$, for a suitable open set $\mathcal{U}$ of $N$.
 This may be achieved by using the \emph{parallel transport} induced by the connection; for this, completeness is essential.

In terms of Lie algebroids, the situation may be described as follows: the differential $\d p:TM\to TN$ is a surjective
 Lie algebroid morphism, whose kernel $\K=\ker \d p$ is given by the tangent spaces to the fibers. In case there exists a
 complete Ehresmann connection, note that local triviality occurs at the level of Lie algebroids since
 $TM|_{\U\times F}=T\mathcal{U}\times TF$ by standard arguments.

 This discussion motivates the following definition:

\begin{defi} A Lie algebroid fibration is a surjective Lie algebroid morphism $A_E \xrightarrow{\pi} A_B$ that covers a submersion
 $E\xrightarrow{p} B$ and admits a complete Ehresmann connection.
\end{defi}

Recall that, given $\pi:A_E \twoheadrightarrow A_B$, an {\em Ehresmann connection of Lie algebroids} is the choice a smooth sub-vector
 bundle $H$ complementary to $\K$ in $A_E$, so $A_E$ splits as a vector bundle as $A_E=\K\oplus H$. Note that such a connection
 always exists. Moreover, it determines a horizontal lift $\sigma:\Gamma(A_B)\to \Gamma(H)$ which is a $C^\infty(A_B)$-linear map,
 $\sigma(\alpha)$ is defined as the unique section of $H$ that $\pi$-projects onto $\alpha$.

An Ehresmann connection is said to be \emph{complete} if $\sharp_E(\sigma(\alpha))$ is a complete vector field for any section
 $\alpha$ such that $\sharp_B(\alpha)$ is complete.

Recall also from \cite[Section 2]{Br} that once a connection has been chosen, we get the analog of a covariant derivative; it is a morphism
 $\Dd: \Gamma(A_B)\to \text{Der}(\K)$ with values in the derivations of $\K$, obtained by letting
 $\Dd_\alpha (\kappa):=[\sigma(\alpha),\kappa]_{A_E}$. Moreover, there is a curvature
 $2$-form $\omega\in \Omega^2(A_B)\otimes \Gamma(\K)$ defined by
 $\omega(\alpha,\beta):=\sigma([\alpha,\beta]_{A_B})-[\sigma(\alpha),\sigma(\beta)]_{A_E}.$
 As a consequence of the axioms of a Lie algebroid structure, the couple $(\Dd,\omega)$ is subject to the following conditions
 \cite[Section 2.1]{Br}:
\begin{eqnarray*}
&\text{Curv}_{\Dd}(\xi_1,\xi_2)(\kappa)=\left[\omega(\xi_1,\xi_2),\kappa\right]_\K\\
&\oint_{\xi_1,\xi_2,\xi_2} \Dd_{\xi_1}\omega(\xi_2,\xi_3)-\omega(\left[\xi_1,\xi_2\right],\xi_3)=0,
\end{eqnarray*}
for any $\xi_i\in\Gamma(A_B),\kappa\in\Gamma(\K)$. Here, $\text{Curv}_{\Dd}(\xi_1,\xi_2)=[{\Dd}(\xi_1),\Dd(\xi_2)]-\Dd_{[\xi_1,\xi_2]}$,
 where we use the bracket of derivations. Therefore we see that, formally at least, the situation is similar to what happens for
 Lie algebras. In fact, the space of $\pi$-projectable sections of $A_E$ is an extension of $\Gamma(A_B)$ by $\Gamma(\K)$
 (see \cite[Lemma 1.9]{Br}). This will play an important role in this paper as we will deeply rely on the use of connections.

\begin{remark}
Note first of all that $E \xrightarrow{p} B$ may not be a locally trivial fibration. Still, it is so orbit-wise, that is, for a small
 open neighborhood $U\subset \mathcal{O}_B$ in an $A_B$-orbit, we have $p^{-1}(\mathcal{U})=F\times \mathcal{U}$ and
 $\K=\K|_F\times \mathcal{U}$ (this is a consequence of Prop. \ref{thm:loc:triv}).
Hence, when $A_B$ is a transitive Lie algebroid, the base map is a fibration and $\K$ is locally trivial.

However, in general $A_E$ will {\em not} split over $\mathcal{U}$ as $\K|_F \times A_B|_\mathcal{U}$, even if $A_B$ is transitive.
 As we already pointed out \eqref{eq:extension-lie-alg}, the simplest counter examples are Lie algebra extensions.
\end{remark}

\begin{remark}
Another way to argue about completeness in our definition of a fibration is by looking at the groupoid level: thinking
 of submersions as the weakest notion of fibration, one may naively try to internalize Lie groupoids into submersions.

Consider the category \underline{Sub} of submersions, where objects are (surjective) submersions $E\xrightarrow{p} B$
 between manifolds, and morphisms are commutative diagrams:
$$\xymatrix{E\ar[r]^{\phi}\ar[d]^{p'}&E'\ar[d]^{p'}\\
            B\ar[r]^{\phi_0} &B}.$$
Then for an  internal groupoid in \underline{Sub}, both the spaces of objects and arrows come as objects in  \underline{Sub},
 so we have: $\G\xrightarrow{\pi}\G_B$, and $E\xrightarrow{p} B$. Moreover, since the groupoid structure maps are morphisms in
 \underline{Sub}, we deduce Lie groupoid structures $\G\tto E $ and $ \G_B\tto B$, for which $\pi$ defines a
 \emph{submersive} groupoid morphism. In other words, by internalizing Lie groupoids into submersions,
 we obtain a notion (seemingly) suitable for Lie groupoid extensions. In fact this construction has two problems.

 Firstly the kernel of $\pi$ may not be a \emph{Lie} subgroupoid. As discussed in
 \cite[Section 2.4]{MK2} (see also \cite{Pradines} for a more general point of view), one shall require that $\pi^!:\G_E\to p^!\G_B$ is also a surjective submersion  (where, $p^!\G_B$
 denotes the usual pull-back groupoid of $\G_B$ along the submersion $p$, and $\pi^!$ the pulled-back morphism of $\pi$).
 Essentially, this condition implies that given any arrow $g\in\G_B$ and any point $e\in E$ lying over the source of $g$,
 there exists $\tilde{g}\in \G_E$ such that $\pi(\tilde{g})=g$, so this corresponds to a \emph{lifting} property. 

Secondly, when looking at the infinitesimal level, we run into similar difficulties: given a Lie algebroid extension
 $\pi:A_E\to A_B$, with $A_E$ and $A_B$ integrable Lie algebroids, the morphism $\pi$ may not even integrate into
 a surjective groupoid morphism if there does not exist a complete connection \cite[Ex 4.11-4.12]{Br}.

By assuming the existence of a complete connection, it is easily seen that both issues are avoided.
\end{remark}

We define the homotopy ``groups'' $\pi_n$  of a Lie algebroid $A$ (see Section \ref{sec:homotopy-groups}) by taking  morphisms
 from spheres $TS^n$ to $A$, then modding  out by homotopies. It turns out that $\pi_{n\ge 2}$ are  bundles of discrete abelian groups over the base manifold $M$, and $\pi_1(A)$ is the topological universal groupoid of $A$
(called Weinstein groupoid in \cite{cf}) of Crainic and Fernandes. We also show the expected long exact sequence for a fibration:

\begin{thm}[Main Theorem] \label{thm:long-exact-seq}
Consider a fibration of Lie algebroids:
$$\xymatrix{A_E\ar[r]^{\pi}\ar[d] &A_B\ar[d]\\
            E\ar[r]^{p} &B,}$$
with fibre $\K\xrightarrow{i} A_E$. Then there is a long exact sequence of groupoids:
\begin{multline}
\label{long:sequence1} \cdots \to \pi_n(\K) \xrightarrow{i_n} \pi_n(A_E) \xrightarrow{\pi_n}p^*\pi_n(A_B) \xrightarrow{ \partial_n} \pi_{n-1} (\K)\to  \cdots \\
\cdots\to p^*\pi_2(A_B)  \xrightarrow{ \partial_2} \pi_{1} (\K) \xrightarrow{i_1} \pi_1(A_E) \xrightarrow{\pi_1} \pi_1(A_B), \end{multline}
where $i_n$ and $\pi_n$ are maps induced by $i$ and $\pi$ respectively, and $\partial_n$ is the transgression map we construct
 in Section \ref{sec:const-trans}.
\end{thm}

In fact, there is a simplicial object $S_\bullet (A)$ associated to any Lie algebroid $A$, with $S_n(A)=\hom (T\Delta^n, A)$,
 the set of Lie algebroid morphisms from the tangent bundle $T\Delta^n$ of the standard $n$-simplex to $A$.
 Then the $n$-th homotopy group $\pi_n(A)$ is just the simplicial homotopy group $\pi_n(S_\bullet(A))$.

 Furthermore,  Theorem \ref{thm:lift} proves that a  Lie algebroid fibration corresponds
 to a \emph{Kan} fibration of the associated simplicial objects. Homotopy theoretic considerations may seem to imply
 that our main theorem  Theorem \ref{thm:long-exact-seq} then follows, however it appears as difficult (if not more) to establish completely a good
 correspondence\footnote{In particular, the smoothness of simplicial objects needs
 a good consideration.}. Hence we leave this discussion to a future work and keep the current work self-contained to  make it
 more accessible to differential geometers and mathematical physicists. 

\noindent {\bf Acknowledgement}: We give warmest thanks to Rui Loja Fernandes for
very helpful discussions and for the invitation for the second author to IST
during which much progress on this work was attained. The first author wishes to
thank the Courant Research Centre ``Higher Order Structures'' for the
invitation during which the work was initiated.

\section{Locally trivial fibrations}

As we explained above, a fibration may not be locally trivial in a naive way.
 There is however a class of fibrations which turns out to be naturally locally trivial: fibrations over a tangent bundle.

This is well known in the case of the Atiyah exact sequences: any transitive Lie algebroid $A\to B$ fits into as an exact sequence:
$$\g_B \hookrightarrow  A\twoheadrightarrow TB, $$
where $\g_B$ is a \emph{locally trivial} bundle of Lie algebras.

Here the base manifolds coincide ($E=B$) but there is a well-known similar result in which the base manifolds differ,
 namely the neighborhood of an orbit in a Lie algebroid. This was first observed in the case of Poisson manifolds by
 Weinstein as the \emph{local splitting Theorem} \cite{We1}, but applies more generally
 to any Lie algebroid (see \cite[Theorem 1.1]{Fe} whose proof was completed by \cite{Dufour1}). This result is easily
 interpreted as a local isomorphism $A\stackrel{_{loc.}}{\sim} T\mathcal{U}\times A_N$, where $\mathcal{U}\subset L$ is an
 open subset contained in an orbit, and $A_N$ the transverse structure at a point $x\in \mathcal{U}$ for an arbitrarily chosen
 (small) transversal $N$.

With both results in mind, it was natural to ask whether any (abstract) fibration over a tangent bundle was necessarily locally
 trivial. The answer is positive:

\begin{thm}\label{thm:loc:triv}
 Let $\pi:A_E\to TB$ be a fibration with base a tangent Lie algebroid. As usual, we denote $\K:=\ker \pi$.
 Then any point $b$ in $B$ admits an open neighborhood $\U$ and an isomorphism of Lie algebroids $\pi^{-1}(T\U)\simeq \K|_{E_{b}}\times T\U$.
\end{thm}
\begin{proof} Here, we will adapt the proof of \cite[Theorem 8.5.1]{DuZu}, the main difficulty being that in our situation,
 $p^{-1}(\mathcal{U})$ is not a neighborhood of $\mathcal{U}$. Let $r$ denote the rank of $\K$ and $p$ the dimension of $B$.
 Then we consider the following property, defined for any $d=0,\dots,p$:
\begin{enumerate}[$(\mathcal{P}_d):$]
  \item \emph{There exists an open neighborhood $\U$ around $b$, coordinates
    $(x_1\dots x_d,\dots,x_p,y_1,\dots,y_q)$ on $p^{-1}(\U)$, a local basis of sections $\kappa_1,\dots,\kappa_r$ of $\K$,
    and a Ehresmann connection $\sigma:\Gamma(TB)\to \Gamma(A_E)$ such that the following relations hold:}
 \begin{eqnarray}
  \left[\sigma_i,\sigma_j\right]&=&0,\hspace{10pt}\hspace{70pt}(i,j=1,\dots,d)         \label{loctriv1}\\
  \left[\sigma_i,\kappa_l\right]&=&0,\hspace{10pt}       \hspace{70pt}  (i=1,\dots,d;l=1,\dots,r)    \label{loctriv2}\\ 
  \sharp_E(\sigma_i)&=&\partial x_i,\hspace{10pt}        \hspace{60pt}(i=1,\dots,d)                         \label{loctriv3}\\
  \left[\kappa_l,\kappa_m\right]&=&{\Sigma}_{s=1}^{r}\, c_{l,m}^{\, s}   \kappa_s,\hspace{10pt}\hspace{22pt}(l,m=1,\dots,r)   \label{loctriv4}\\
  \sharp(\kappa_s)&=&\Sigma_{v=1}^q K_{s,v} \partial y_v,\hspace{10pt} \hspace{16pt}(s=1,\dots,r)\label{loctriv5}
  \end{eqnarray}
\emph{Here, we denoted $\sigma_i:=\sigma(\partial x_i)$ and $c_{l,m}^{\,s}, K_{s,v}$ are smooth functions required to depend only on the variables $(x_{d+1},\dots,x_p,y_1,\dots,y_q)$.}
\end{enumerate}
Note that when $d=p$, we get that $\sigma(\partial x_i)$ $(i=1,\dots,p)$ mutually commute, that they commute with
 $\kappa_l$ $(l=1,\dots,r)$ and that the brackets $[\kappa_l,\kappa_m]$ and the anchor of $\K$ are independent of $x$;
 thus the map $\sigma(\partial x_i)\to \partial x_i, \kappa_s\to \kappa_s$ provides the desired isomorphism.

In order to prove $(\mathcal{P}_p)$, we proceed by induction on $d$.  When $d=0$, there is nothing to show so may assume that $(\mathcal{P}_d)$
 is satisfied for $d\ge 0$. Then we can denote  $\omega=\sum_{i<j} \omega_{i,j}^{\, l} \kappa_l\otimes \d x_i\wedge \d x_j$ the curvature $2$-form, where $\omega_{i,j}^{\, l}$ are smooth functions of $(x,y)$. Note that by induction, $\omega_{i,j}^{\, l}$ already vanishes when both $i,j\leq d$.

We first have to focus on $\omega^l_{i,d+1}$, which we will arrange to vanish for all $i=1,\dots,d$ by modifying the connection. On the one hand we have, by definition of $\omega$:
$$[\sigma_i,\sigma_{d+1}]=\sum_{l=1,\dots,r} \omega_{i,d+1}^{\, l} \kappa_l, \quad\quad(i=1,\dots,d).$$
On the other hand, a computation shows that the equation $\partial_H\omega(\partial x_i,\partial x_j,\partial x_{d+1})=0$  takes the form:
$$\frac{\partial \omega^{\, l}_{j,d+1}}{\partial x_i}-\frac{\partial \omega^{\, l}_{i,d+1}}{\partial x_j}=0, \quad\quad (i,j\leq d).$$
By applying the Poincar\'e Lemma for any $l=1,\dots,r$, one gets functions $\psi^l$ satisfying:
$$-\frac{\partial \psi^l}{\partial x_i}=\omega_{i,d+1}^{\, l},       \quad\quad(i=1,\dots,d).$$
If we let $\Delta:=(\sum_{l=1}^r \psi^l \kappa_l)\otimes \d x_{d+1}\in \Omega^1(B)\otimes\Gamma(K)$ and consider the Ehresmann
 connection with horizontal lift given by $\sigma'=\sigma+\Delta$, we know from \cite{Br} that the corresponding curvature will be
 $\omega'=\omega+\partial_H\Delta + [\Delta\wedge\Delta]$. In particular, we can check that
$$\omega'(\partial x_i,\partial x_{d+1})=\sum_{l=1}^r  \omega_{i,d+1}^l \kappa_l+ \sum_{l=1}^r\frac{\partial \psi^l}{\partial x_i}\kappa_l=0.$$
Thus we have proved \eqref{loctriv1} in $(\mathcal{P}_{d+1})$.
 Now by applying the anchor map, we immediately get that $[\sharp_E(\sigma'(\partial x_i)),\sharp_E(\sigma'(\partial x_{d+1}))]=0$,
 so we can choose coordinates for which $\sharp_E \sigma'(\partial x_{d+1})=\partial x_{d+1}$.

We now want to replace $\kappa_1,\dots,\kappa_r$ by $\tilde{\kappa}_l=\sum g_{l,m}\kappa_m$,
 where $g_{l,m}$ are functions on $p^{-1}(\U)$, in such a way that $\eqref{loctriv2}$ is satisfied in $(\mathcal{P}_{d+1})$.
 For this we need that:
\begin{enumerate}[i)]
 \item $[\sigma'(\partial x_{d+1}),\tilde{\kappa}_m]=0$ is satisfied;
 \item $g_{l,m}$ is independent of $x_i$ for all $i=1,\dots, d$.
\end{enumerate}
 Condition i) is equivalent to the following system of equations:
\begin{equation}
\label{eqstep2}\frac{\partial g_{l,m}}{\partial x_{d+1}}+\sum_{u=1,\dots,r} g_{l,u}f_{d+1,u}^m=0, \quad\quad(l,m=1,\dots,r)
\end{equation}
where the functions $f_{d+1,u}^m$ are given by:
$$[\sigma'(\partial x_{d+1}),\kappa_u]=\sum_{m=1,\dots,r} f_{d+1,u}^m \kappa_m.$$
We know that the equations \eqref{eqstep2} have a unique solution $(g_{l,m})$ with initial condition
 $$g_{l,m}|_{\{x_{d+1}=0\}}=\delta_{l,m}.$$
 Moreover it is easily seen that this solution is independent of $x_i$ 
 provided the initial condition and $f_{d+1,u}^m$ are, thus ensuring ii). 
 To see that  $f_{d+1,u}^m$ is indeed independent of $x_i$ (for $i=1 ,\dots, d$) one may write the Jacobi identity involving
 $\sigma'(\partial x_{d+1}),\sigma'(\partial x_i),\tilde{\kappa}_u$ and use the fact that 
$[\sigma'(\partial x_{d+1}),\sigma'(\partial x_i)]=0$ and $[\sigma'(\partial x_{d+1}),\tilde{\kappa}_u]=0$.

In order to conclude, we have to check that \eqref{loctriv4} in $(\mathcal{P}_{d+1})$ is satisfied as well.
 This is just a consequence of the derivation property: $$[\sigma'(\partial x_{d+1}),[\tilde{\kappa}_l,\tilde{\kappa}_m]]
=[[\sigma'(\partial x_{d+1}),\tilde{\kappa}_l],\tilde{\kappa}_m]+[\tilde{\kappa}_l,[\sigma'(\partial x_{d+1}),\tilde{\kappa}_m]],$$
written in local coordinates. The last equation \eqref{loctriv5} is obvious and the independence condition
 follows by applying the anchor to $[\sigma'(\partial x_i),\kappa_s]=0$.

Now the above proof is not totally satisfactory since we had to choose a \emph{local} basis of sections of $\K$,
 \emph{i.e.} the $\kappa_l$'s are only defined on a small neighborhood $\mathcal{V}\subset p^{-1}(\mathcal{U})$.
 In general, even the restriction $\K|_{E_b}$ may not admit global sections, so we need a proof that make no reference to the
 $\kappa_l's$. In order to achieve this, one may proceed as follows:
\begin{itemize}
\item first we observe that $\K|_{p^-1(\mathcal{U})}$ identifies with $\K|_{E_b}\times \mathcal{U}$ as a vector bundle
 (this follows from the completeness assumption). Moreover, one can still express the curvature form as
 $\omega=\sum_{i\leq j}\omega_{i,j}\d x_i\wedge \d x_j$, with now $\omega_{i,j}\in\Gamma(\K)$,

\item then we apply a similar induction process as above, replacing the condition \eqref{loctriv2} by 
$[\sigma_i,\kappa]=\frac{\partial \kappa}{\partial x_i}$ for any
 $\kappa\in\Gamma(\K|_{p^{-1}(\mathcal{U})})\simeq \Gamma(\K|_{E_b}\times \mathcal{U})$.
 It is easily checked using Leibniz rule that this condition implies $\sharp_E(\sigma_i)=\partial x_i$.
 Moreover it makes it easy to apply Poincar\'e lemma in the first step of the induction,

\item in the second step, we have to prove that $[\sigma'_{d+1},\kappa]=\frac{\partial \kappa}{\partial x_{d+1}}$ for a good
 identification of $\K|_{p^{-1}(\mathcal{U})}$. This becomes obvious once observed that the bracket with
 $\sigma_1,\dots,\sigma_d,\sigma'_{d+1}$ induces \emph{commuting} linear vector fields on $\K|_{p^{-1}(\mathcal{U})}$,

\item to conclude, we notice in equations \eqref{loctriv4} \eqref{loctriv5} that the independence of the structure
 functions on the variables $x_1,\dots,x_{d+1}$ can be replaced by the following conditions:
 $\frac{\partial}{\partial x_i}[\kappa,\kappa']=0, \frac{\partial}{\partial x_i}\sharp_\K(\kappa)=0$ whenever
 $\kappa,\kappa'\in\Gamma(\K)$ are independent of $x_1,\dots,x_{d+1}$. As above, we just have to apply Jacobi identity,
 and use the anchor map.
\end{itemize}
\end{proof}

\section{Homotopy groups} \label{sec:homotopy-groups}
In this section, we want to define a nice notion of homotopy groups for Lie algebroids. For the sake of computations,
 it is more convenient to define usual spheres as morphisms $\gamma:I^n\to M$ whose restriction to the boundary is reduced to a
 point $m_0$; the reason for this is that $I^n$ admits simple parametrizations. If we differentiate $\gamma$, we see that this condition is
 equivalent to require that $\d\gamma:TI^n\to TM$ vanishes on $T\partial I^n$.

One can treat similarly the notion of homotopy: a family of spheres $\gamma^{t_{n+1}}: I^n\to M$, $(t_{n+1}\in I)$ is a homotopy
 if and only if the base points $m_0(t_{n+1}):=\gamma^{t_{n+1}}(\partial I^n )$ do not depend on $t_{n+1}$. By differentiation,
 this condition is equivalent to requiring that ${\frac{\partial \gamma}{\partial t}}_{\!_{n+1}}$ vanishes whenever one of
 $\{t_1,\dots,t_n\}$ equals $0$ or $1$. With this in mind, one may define homotopy groups of a Lie algebroid as follows:
\begin{defi}\label{def:sphere}
A $n$-dimensional $A$-sphere is a morphism $a=\sum a_k \d t_k:TI^n\to A$ such that $a_k(t_1,\dots, t_n)=0$ if one of
 $\{t_1,\dots,\hat{t}_k,\dots,t_n\}$ is $0$ or $1$. We denote the set of $n$-dimensional $A$-spheres by $\cS_n(A)$. 
\end{defi}

\begin{defi}\label{def:homotopy}
Two $n$-dimensional spheres $a_0, a_1$ are homotopic if there exists a Lie algebroid morphism $h=\sum h_k \d t_k:TI^{n+1}\to A$ 
satisfying the following properties:
\begin{Itemize}
\item $a_\epsilon =\sum_{k=1}^n h_k(t_1,\dots,t_n,\epsilon)\d t_k$ for $\epsilon=0,1$;
\item $h_{n+1}$ vanishes whenever one of $\{t_1,\dots,t_n \}$ is $0$ or $1$.
\end{Itemize}\end{defi}

This notion of homotopy defines an equivalence relation on $\mathcal{S}_n(A)$ since, as we will explain in Section
 \ref{sec:concatenation}, homotopies can be concatenated.

\begin{defi}
For any Lie algebroid $A$ and any integer $n\geq1$, the $n$-th homotopy group of a Lie algebroid, denoted $\pi_n(A)$, is defined as the space of equivalence
 classes of $n$-spheres up to homotopy.
\end{defi}

In this definition, the word ``group'' is slightly misleading. In fact, when  $n=1$, Definition
\ref{def:sphere}, \ref{def:homotopy} simply yield $A$-paths,
and homotopies of $A$ paths as in \cite{cf}. Therefore, $\pi_1(A)$ coincides with the source simply connected topological
 groupoid $\G(A)$ (possibly Lie). In particular $\pi_1(A)$ is a \emph{groupoid} rather than a group (see Prop. \ref{prop:pi1}). 

For $n\geq 2$, $A$-spheres are naturally based at one point which is independent of the homotopy class, so we have a surjection
 $\pi_n(A)\to M$. To see this, one may compose spheres and homotopies with the anchor map; then everything follows from the
 discussion at the beginning of this section (see also Example \ref{ex:tangent:algebroid}). 

We denote $\pi_n(A,x)$ the homotopy classes of spheres based at $x\in M$ ($n\geq 2$). As we shall see, $\pi_n(A)$
 for $n>1$ is a \emph{bundle of groups} over $M$ rather than a group. When $n=1$,
 we will say by convention that an $A$-path $a$ covering $\gamma:I\to M$ is based at $\gamma(0)$. 

\begin{ep}[$A=TL$] \label{ex:tangent:algebroid} The  first observation is that the algebroid homotopy group
$\pi_n(TL)=\pi_n(L)$ because  $TS^n \to TL$ is an algebroid morphism if and only if it
 is the tangent map $T\gamma$ of a map   $\gamma:S^n \to L$.
\end{ep}

\begin{ep}[$A=\g$] \label{ex:g} When $A$ is simply a Lie algebra, $\pi_0(\g)=1$
because the base of $\g$ is trivial. $\pi_1(\g)=G$ the simply connected Lie
group integrating $\g$ \cite{cf,lie-group,van-est}. When $n\ge 2$, we have $\pi_{n}(\g)=\pi_n(G)$.
 In fact, any algebroid morphism $TS^n \to \g$ corresponds by Lie's II theorem of Lie algebroid to a Lie groupoid morphism $S^n
\times S^n \to G$, where $S^n \times S^n$ is given the pair groupoid
 structure over $S^n$. We can do so because $S^{n \ge 2}$ is simply
connected. But any groupoid morphism $\Phi:S^n \times
 S^n\to G$ is equivalent to a map $\phi:S^n\to G$ by defining $\Phi(a,
 b)=\phi(a)\cdot \phi(b)^{-1}$. Thus $\pi_n(\g)=\pi_n(G)$ for $n\ge 2$. 

Another way to see this is to analyze explicitly the Lie algebroid morphisms from $TI^n$ to $\g$. A bundle map $a:TI^n\to \g$
 writes $a=\sum a_k \d t_k$, where $a_k:I^n\to \g$. Then $a$ defines an algebroid morphism \emph{iff} the following relations are satisfied:
$$\frac{\d a_i}{\d t_j}-\frac{\d a_j}{\d t_i}=[a_i,a_j].$$
If we view $a$ as a $\g$-valued connection 1-form of the trivial principal bundle $G\times I^n \to I^n$, then this equation
 tells us exactly that $a$ is flat. On the other hand, if one thinks of $a_i+\partial t_i$ as left invariant vector fields
 on $G\times I^n$, these equations mean that they commute with each other.
 Therefore, $a_i$ may be seen as a $t_i$-time dependent family of vector fields, with parameters $(t_1,\dots,\hat{t}_i,\dots,t_n)$.

Then it is easily seen that the condition for $a$ to be a sphere implies that the image of the identity $1_G$ by the time dependent
 flow $\phi^{a_1}_{t_1,0}$ is a usual sphere $I^n\to G$ whose restriction to $\partial I^n$ is reduced to $1_G$.

It will be useful for our purposes to think this way: in order to work with a morphism $a:TI^n\to A$, where $A$ is an arbitrary
 Lie algebroid, we will extend the $a_i$'s into time-dependent sections $\alpha_i$'s satisfying the above relations
 (see Proposition \ref{prop:morphisms:sections}). Then we can treat $\alpha_i$'s as (right invariant) vector fields on $\pi_1(A)$
 when it is a Lie groupoid. In general, when $\pi_1(A)$ fails to be a Lie groupoid, this treatment is still valid if we replace
 it by $\cH(A)$, the stacky Lie groupoid constructed in \cite{tz}. A similar argument as in the case of groups tells us that
 $\pi_{n\ge 2} (A, x_0)= \pi_{n\ge 2}(s^{-1}(x_0), x_0)$ where $s$ is the source of $\cH(A)\xrightarrow{s} M$. 
 For instance, if $a$ is a sphere based at $x_0$, the boundary conditions easily ensure that one obtains a usual sphere
 in the source fiber over $x_0$ simply by flowing the identity at $x_0$.
\end{ep}

\begin{remark}\label{rem:squares:vs:spheres} When $n\ge 2$, our approach is equivalent to the one that would define spheres as
 pointed morphisms $(TS^n, N) \to (A, x)$ and homotopies as morphisms $f:(TD^n, N) \to (A, x)$, where $N\in S^n$ is the north
 pole and $x$ is a point in the base of $A$.

To show this, one may proceed as follows. First, by showing that any morphism $TS^n\to A$ is homotopic to one that vanishes on
 $T_N S^n$. For this we consider $H:S^n\times I \to S^n$ a smooth map such that
\begin{Itemize}
\item $H|_{S^n\times\{0\}}=\id_{S^n}$;
\item $H|_{S^n\times\{1\}}=\phi$;
\item $H(N,\epsilon)=N$ for any $\epsilon\in I$.
\end{Itemize}
Here $\phi:S^n\to S^n$ is a smooth map with $\phi(N)=N$, $\d \phi_N=0$ and restricting to a diffeomorphism on $S^n-\{N\}$.
 By compositing with $\d H$, we see that any morphism $a:TS^n\to A$ is homotopic to $a\circ \d\phi: TS^n\to A$,
 the latter vanishing on $T_N S^n$.

Secondly, a morphism $TS^n\to A$ that vanishes when restricted to $T_NS^n$ is the same thing as a morphism $c:TI^n\to A$ that vanishes over ${\partial I^n}$.
 This can be seen by considering a smooth map $j:I^n\to S^n$ with  $j(\partial I^n)=N$ and $j|_{I^n-\partial I^n}$ a diffeomorphism onto  $S^n-\{N\}$.

Now, a morphism $TI^n\to A$ that vanishes on $TI^n|_{\partial I^n}$ is not the same as one that vanishes on $T\partial I^n$. However this is true up to homotopy: consider a cut-off function that is, a smooth function $\tau: \R\to \R$ such that, 
\begin{equation}\label{eq:cut-off} \text{$\tau_{\{t\leq 0\}}=0$, $\tau_{\{t\geq 1\}}=1$ and $\tau'(t)>0$ for $t\in]0,1[$.}
\end{equation} Then we define $h:I^{n+1}\to I^n$ by:
$$h(t_1,\dots,t_{n+1}):=\bigl((1-\tau(t_{n+1}))t_1+\tau(t_{n+1})\tau(t_1),\dots,(1-\tau(t_{n+1}))t_n+\tau(t_{n+1})\tau(t_n)\bigr),$$ and simply use composition with $\d h$ to show that any sphere $a:TI^n\to A$ is homotopic to its reparametrization $a\circ \d r^\tau$, where $r^\tau(t_1,\dots,t_n)=(\tau(t_1),\dots,\tau(t_n))$. Then it is easily seen that $a\circ \d r^\tau$ vanishes over $\partial I^n$.
\end{remark}

\subsection{Operations on cubes}\label{sec:concatenation}

As we have already pointed out, we will work with morphisms $TI^n\to A$ rather than the usual simplices $T\Delta^n\to A$.
The corresponding notion of a simplicial set is called a \emph{cubical} set. Define 
 $X_n(A):=\{\text{Lie algebroid morphisms } TI^n\to A\}$. We will be more precise on the notion of faces and degeneracies
 on $X_n(A)$ since we think it is important for understanding $A$-spheres.

First we denote:
\begin{Itemize}
 \item $i_{p,\epsilon}^k: I^{k-1}\to I^k$ the $p$-th injection: $i_{p,\epsilon}^k(t_1,\dots,\hat{t_k},\dots, t_n):=(t_1,\dots, t_{p-1},\epsilon,t_{p+1},\dots,t_n)$,
 \item $\pi_p^k:I^{k+1}\to I^k$ the $p$-th projection: $\pi_p^k(t_1,\dots,t_{k+1}):=(t_1,\dots,\hat{t}_p,\dots, t_{k+1})$
\end{Itemize}
Then on $X_n(A)$, we define:
\begin{Itemize}
\item $2k$ faces maps $d_{p,\epsilon}^k:X_k\to X_{k-1}$ given by:
$d^k_{p,\epsilon}(a):=a\circ \d i_{p,\epsilon}^k\,(p=1,\dots,k; \epsilon=0,1)$
\item $k+1$ degeneracies $s_p^k:X_k\to X_{k+1}$ given by:
$ s_p^k(a):=a\circ \d\pi_p^k\,(p=1,\dots,k+1).$
\end{Itemize}
More explicitly, if we denote $a=\sum_{l=1}^k a_l \d t_l\in X_n(A)$ the components of $a$, then the faces and degeneracies of $a$
 have the following components:
\begin{eqnarray*}
       d_{p,\epsilon}^k(a)(t_1,\dots,t_n)&=&\sum_{l\neq p} a_l(t_1,\dots,t_{p-1},\epsilon,t_{p+1},\dots,t_n)\d t_l\\
       s_p^k(a)(t_1,\dots,t_{n+1})&=&\sum_{l< p}a_l(t_1,\dots,\hat{t}_p,\dots,t_{k+1})\d t_l+\sum_{p<l}a_l(t_1,\dots,\hat{t}_p,\dots,t_{k+1})\d t_{l+1}.
\end{eqnarray*}
As for usual simplices, there are obvious coherence conditions satisfied by faces and degeneracies,
 we will not make them precise since we will not be needing them. Rather, note that by definition $c\in X_n(A)$ is a sphere
 \emph{iff} all its faces are trivial.

Suppose now that we are given two morphisms $a^i:TI^n\to A$ satisfying $d^n_{n,1}a^0=d^n_{n,0}a^1 $.
 Then we define the concatenation $a^1\odot_{t_n} a^0$ as follows:
$$a^1\odot_{t_n} a^0=
\left\{\begin{array}{ccc}  (a^0)^\tau\circ\d u_0,  &\text{ if } t_n\in [0,\frac{1}{2}], \\ 
                           (a^1)^\tau\circ \d u_1        &\text{ if } t_n\in [\frac{1}{2},1].
\end{array}\right.$$
In this expression:
\begin{itemize}
 \item $(a^i)^{\tau}$ denotes the re-parametrization of $a^i$ along the $n$-th coordinate by a cut-off function $\tau$
 as in \eqref{eq:cut-off}, that is, $(a^i)^\tau$ is the composed morphism $TI^n\xrightarrow{dr_n^\tau} TI^n\xrightarrow{a}A$,
 with $r_n^\tau(t_1,\dots,t_n):=(t_1,\dots,t_{n-1}, \tau(t_n))$;
\item the applications $u_i, i=0,1$ are given by:
$$u_0(t_1,\dots,t_n)=(t_1,\dots,t_{n-1}, 2 t_n),\quad  u_1(t_1,\dots,t_n)=(t_1,\dots,t_{n-1}, 2t_{n}-1).$$\end{itemize}
Note here that the composition $(a^i)^\tau\circ\d u_i$ takes into account the base maps, so it covers
 $\gamma_i\circ r_n^\tau\circ u_i$. In fact, let $g_{t_n}$ be the composed map:
\begin{equation} \label{eq:glue-map}
TI^n \to T(I^n\sqcup_{I^{n-1}} I^n) \hookrightarrow TI^n \sqcup TI^n,
\end{equation}
where the first map is the tangent map of 
$$
I^n \to I^n\sqcup_{I^{n-1}} I^n, \quad \left\{
\begin{array}{ccc} &(t_1,\dots,t_n) \mapsto (t_1,\dots,t_{n-1},\tau(2t_n))\;\text{in the right}\;I^n,&\text{if $t_n\in [0, \frac{1}{2}]$,}  \\  
&(t_1, \dots, t_n) \mapsto (t_1, \dots, t_{n-1}, \tau(2t_n-1))\; \text{in the left}\;I^n , & \text{if $t_n\in [ \frac{1}{2}, 1]$} 
\end{array} \right.
$$ Then the resulting map $a^1\odot_{t_n} a^0$ is the composition 
\begin{equation}\label{eq:a-glue} 
a^1\odot_{t_n} a^0: TI^n \xrightarrow{g_{t_n}} TI^n \sqcup TI^n \xrightarrow{a^1 \sqcup a^0} A.  
\end{equation}

\begin{lemma}
 The concatenation $a^1\odot_{t_n} a^0$ is smooth, provided that $a^0,a^1$ are.
\end{lemma}
\begin{proof} 
In order to concatenate $a^0=\sum a^0_k\d t_k$, $a^1=\sum a_k^1\d t_k:TI^n\to A$ along the $n$-th coordinate,
 recall that the composability assumption $d_{n,1}^na^0=d_{n,0}^n a^1$ writes $a^0_k|_{\{t_n=1\}}=a^1_k|_{\{t_n=0\}}$
 for all $k=1,\dots, n-1$. Moreover, the reparametrizations are given by:
\begin{eqnarray*}
  (a^0)^\tau\circ\d u_0
&=&\sum_{k<n}a^0_k(t_1,\dots,t_{n-1},\tau(2t_n))\d t_k+2\tau'(2t_n)a^0_n(t_1,\dots,t_{n-1},\tau(2t_n))\d t_n,\\
  (a^1)^\tau\circ\d u_1
&=&\sum_{k<n}a^1_k(t_1,\dots,t_{n-1},\tau(2t_n-1))\d t_k+2\tau'(2t_n-1)a_n^1(t_1,\dots,t_{n-1},\tau(2t_n-1))\d t_n.
 \end{eqnarray*}

Smoothness for the $n$-th coordinate follows from $\displaystyle\lim_{t_n\to 1/2}((a^0)^\tau\circ \d u_0)_n=\lim_{t_n\to 1/2}((a^1)^\tau\circ \d u_1)_n=0$ as well as all the derivatives. Then we argue the smoothness of the $k$-th component ($k<n$) as follows:
\begin{itemize}
\item continuity is clear because of the composability assumption,
\item we notice that $a^0_k|_{\{t_n=1\}}=a^1_k|_{\{t_n=0\}}$ implies
 $\frac{\partial a^0_k}{\partial t_l}|_{\{t_n=1\}}=\frac{\partial a_k^1}{\partial t_l}|_{\{t_n=0\}}$ for any
 $l=1,\dots,n-1$. Therefore, since $a_k^i$ is smooth, we get:
$$\lim_{t_n\to 1}\frac{\partial a^0_k}{\partial t_l}=a^0_k|_{\{t_n=1\}},\quad\quad
  \lim_{t_n\to 0}\frac{\partial a^1_k}{\partial t_l}=a_k^1|_{\{t_n=0\}}.$$
On the other hand, a direct computation shows that:
$$\displaystyle{\lim_{t_n\to 1}}\frac{\partial {a_k^0}^\tau}{\partial t_l}=\lim_{t_n\to 1}\frac{\partial a_k^0}{\partial t_l},\quad\quad
\displaystyle\lim_{t_n\to 1}\frac{\partial {a^1_k}^\tau}{\partial t_l}=\lim_{t_n\to 1}\frac{\partial a^0_k}{\partial t_l},$$
where ${a_k^i}^\tau$ denote the components of $(a^i)^\tau$.
We conclude that $\displaystyle\lim_{t_n\to 1}\frac{\partial {a^0_k}^\tau}{\partial t_l}
=\lim_{t_n\to 0}\frac{\partial {a^1}_k^\tau}{\partial t_l}$, which ensures the concatenation is of class $C^1$;
\item for $l=n$, it is easily seen that $\displaystyle\lim_{t_n\to 1}\frac{\partial {a^0_k}^\tau}{\partial t_n}=\displaystyle\lim_{t_n\to 0}\frac{\partial {a^1_k}^\tau}{\partial t_n}=0$;
\item for higher derivatives, the argument is the same. Since $a^0_k|_{t_n=1}=a^1_k|_{t_n=0}$, all their derivatives with respect
 to $t_l, t_m$ for $l,m=1, \dots, n-1$ coincide too.
\end{itemize}\end{proof}

We have constructed a concatenation for spheres based at a same point along $t_n$. We  define similarly  concatenations $a_1\odot_{t_i} a_0$ along $t_i$. Since the maps $g_{t_i}$ and $g_{t_j}$ are homotopic to each other by the standard treatment of usual homotopy groups of topological spaces, the following lemma is immediate:
\begin{lemma}\label{lemma:diff-glue}
Concatenation along different parameters gives homotopic spheres.
\end{lemma}
\subsection{Properties of the homotopy groups}
We can now state the first properties of homotopy groups. First of all, it is easy to see from the definitions that 
\begin{prop}\label{prop:pi1}
Given a Lie algebroid $A$, the first homotopy group $\pi_1(A)$ is the topological
Weinstein groupoid of $A$ as in \cite{cf}.
\end{prop}

\begin{prop} For any $n\ge 2$, $\pi_{n} (A)$ is a bundle of groups.
\end{prop}
\begin{proof}
It is clear that the concatenation defined in \ref{sec:concatenation} allows to concatenate any couple of spheres based at a
 same point. The only thing we have to check is that any morphism $a:TI^n\to A$ is homotopic to its reparametrization $a^\tau$,
 in order to make sure that the composition law is well-defined. This is easily achieved by composing $a$ with $\d h_n^\tau:TI^{n+1}\to TI^n$, where:
$$h_n^\tau(t_1,\dots, t_{n+1}):=(t_1,\dots, t_{n-1}, (1-\tau(t_{n+1}))t_n+\tau(t_{n+1})\tau(t_n)).$$
\end{proof}

\begin{prop} The fibre of $\pi_{n\ge 2} (A)$ is abelian.
\end{prop}
\begin{proof}
By a reasoning similar to Remark \ref{rem:squares:vs:spheres} one can transport the usual proof for $\pi_n(S^n)$ and use composition at the source to show that $a_1\odot a_0$ is homotopic to $a_0\odot a_1$.
\end{proof}

Our main theorem is Theorem \ref{thm:long-exact-seq} stated in the
introduction.  Notice that a Lie algebroid morphism $f:TI^n \to A$ will stay within an
orbit, that is to say that there is an orbit $L$ such that $f$ is just
$f: TI^n \to A|_L$. Thus $\pi_n(A)|_L=\pi_n(A|_L)$. In particular, $\pi_n (A)$ has the same fibre as $\pi_n(A|_L)$
at a point $x\in L$. Hence to study $\pi_n(A)$, we only have to study
$\pi_n(A|_L)$. The following corollary will tell us that the homotopy group $\pi_n(A)$ is
determined by the homotopy groups of the leaves and isotropy groups.

\begin{cor}\label{cor:A-L}
Let $L$ be an orbit (hence connected) of $A$, then we have a fibration $ A|_L
 \xrightarrow{\sharp} TL$ of Lie algebroids with the fibre
 $\ker\sharp|_L$ a locally trivial bundle of Lie algebras with fiber type $\g$. We denote $G$ the simply connected Lie group
 integrating $\g$.

 By Theorem \ref{thm:long-exact-seq} and Prop. \ref{prop:pi1}, we obtain a long exact sequence
\begin{equation}\label{eq:A-L}
\begin{split}
 &\dots \to \underline{\pi_n(G)}\to \pi_n(A|_L) \to \pi_n(L) \to \underline{\pi_{n-1}(G)} \to \\
 &\dots \to   \pi_2(L) \xrightarrow{\partial_2} \cG(\ker \sharp) \to \G(A|_L) \to \G(TL) \to 1
\end{split}
\end{equation}
where $\underline{\pi_n(G)}$ is a locally trivial bundle over $L$ with
fibre $\pi_n(G)$, and $\cG(-)$ is the topological
Weinstein groupoid constructed in \cite{cf} of the corresponding Lie
algebroid. Moreover
\begin{enumerate}
 \item \label{itm:pi} $\pi_2(A|_L)=\ker \partial_2$;
\item \label{itm:partial} Point-wisely, $\partial_2$ is  exactly the map $\partial$ in \cite[Prop. 3.5]{cf}, thus $\im \partial_2$
 is a bundle of groups whose fibre at $x$ is the monodromy group $\tilde{\mathcal{N}}_x(A)$ in \cite[Def. 3.2]{cf},
 which controls the integrability of $A$. 
\end{enumerate}
\end{cor}
\begin{proof}
Any splitting of $A|_L \xrightarrow{\sharp}L$ gives us an Ehresmann
connection and it is automatically complete by the definition of
completeness. By Example \ref{ex:tangent:algebroid}, when $n\ge 2$, we have
$\pi_{n}(TL) =\pi_{n}(L)$ as a fibre bundle over $L$. 
Note that $\ker \sharp$ is a locally trivial Lie algebra bundle over
$L$, so locally $\pi_n(\ker \sharp)|_{\mathcal{U}}=\pi_n(G) \times
\mathcal{U}$ is a constant fibre bundle over $L$ when $n\ge 2$ (see Example \ref{ex:g}).

Then item \eqref{itm:pi} follows from a part of the long exact sequence \eqref{eq:A-L}
 \[
 \dots \underline{\pi_2(G)}\to \pi_2(A|_L) \to \pi_2(L) \xrightarrow{\partial_2} \G(\ker\sharp) \to \dots
\]and the fact that $\pi_2(G)=0$ for any finite dimensional Lie group.

The transgression map $\partial_2$ is constructed in Section
\ref{sec:const-trans} by lifting the sphere, and then restricting on a
certain boundary. In the context there, if we take $a_1$ to be $a$, $a_2$ to be $b$, $t_1$ to
be $t$, and $t_2$ to be $\epsilon$,  it is easy to see that  our
construction of $\partial_2: \pi_2(L) \to \cG(\ker \sharp)$ gives the
construction of $\partial$ in \cite[Prop. 3.5]{cf}. 
\end{proof}

\begin{cor}\label{cor:count} For $n\ge 2$, $\pi_{n} (A)$ has countable fibres. Moreover
restricting on an orbit $L$,   $\pi_n(A)|_L=\pi_n(A|_L)$ is  an \'etale bundle over $L$.
\end{cor}
\begin{proof} The fact that $\pi_n(A)$ has countable fibres follows directly from Cor. \ref{cor:A-L} since both
  $\pi_n(G)$ and $\pi_n(L)$ are countable. The space of $\pi_n(A)$ has a natural
  topology induced from mapping spaces. Since the fibre of
  $\pi_n(A)$ is countable it must be discrete. Then the same proof to show that  the usual homotopy groups $\sqcup_x \pi_n(M, x)$ form a locally trivial fibration over $M$ shows that $\pi_n(A)|_L$ is also a local trivial fibration. Since $\pi_n(A)|_L$ has discrete fibres, $\pi_n(A)|_L$ is an \'etale bundle. 
\end{proof}

Take a linear splitting $\sigma: TL \to A_L$ of the anchor $\sharp$. The curvature of $\sigma$ is the element $\Omega_\sigma\in \Omega^2(L; \ker \sharp)$ defined by 
\[ \Omega_\sigma(X, Y):= \sigma([X, Y]) - [\sigma(X), \sigma(Y)]. \]
Then combining with \cite[Lemma 3.6]{cf} and Cor. \ref{cor:A-L}, we have:

\begin{cor} \label{cor:compute} Let $L$ be the orbit at $x$, then 
$\pi_2(A)_x=\{[\gamma]\in \pi_2(L, x): \int_\gamma \Omega_\sigma =0\}$ if $\Omega_\sigma$ takes values in the center of $\ker\sharp$.
\end{cor}


\subsection{Examples from Poisson geometry}

The transgression maps are difficult to compute explicitly in general (in fact, at least as difficult as for usual fibrations).

However, in the case $\K$ is a bundle of abelian Lie algebras, the second transgression $\partial_2$
 is rather easy to describe since it is given by an integration as we see in Cor. \ref{cor:compute} for the fibration
 $\ker \sharp \to A|_L \to TL$. We will explain this  in the next examples.
\begin{ep}
 Recall that a Poisson manifold $(M,\Pi)$ can be considered as a Jacobi manifold \cite{CZ}; the associated Lie algebroid structure
 lies on $\mathbb{R}\oplus T^*M\to M$ rather than on $T^*M$, and fits into a fibration of Lie algebroids,
$$\mathbb{R}\times M\hookrightarrow \mathbb{R}\oplus T^*M\twoheadrightarrow T^*M.$$
So there is a map $\partial_2:\pi_2(T^*M)\to \G(\mathbb{R}\times M)\simeq\mathbb{R}\times M$.

On the other hand, one may apply Cor. \ref{cor:A-L} to identify $\pi_2(T^*M)$: for any symplectic leaf $L$ in $M$,
 we get a fibration:
$$ \ker\Pi^\sharp|_{L}\hookrightarrow T^*M|_{L}\twoheadrightarrow TL,$$
and we deduce a transgression map $\partial'_2:TL\to \G(\ker \Pi^\sharp|_L)$,
 so that $\pi_2(T^*M)|_L=\pi_2(T^*M|_L)=\ker\partial_2'=\{S\in\pi_2(TL)| \partial_2'S=1\}$.

The image of $\partial_2$ can be obtained by integrating the symplectic form $\omega_L$ over elements
 of $\pi_2(T^*M)\subset\pi_2(TL)$ (see \cite[Lemma 4.4]{CZ}). 

 We know \cite{Br} that $\im \partial_2$ measures the integrability of $\mathbb{R}\oplus T^*M$ provided $T^*M$ is integrable,
 so we recover this way the result of \cite{CZ}.
 Note also that $\pi_2(\mathbb{R}\oplus T^*M)|_{L}=\{S\in\pi_2(TL)|\partial_2'S=1,\partial_2S=0\}$.
\end{ep}

\begin{ep} More generally, one may consider a Poisson structure $\Pi$,  and a bi-vector $\Lambda \in \Gamma(\Lambda^2T^*M)$
 satisfying $[\Lambda,\Pi]=0$. Then seeing $\Lambda$ as a $2$-cocycle for the trivial representation, we obtain a fibration:
$$M\times \mathbb{R}\hookrightarrow \mathbb{R}\rtimes_{\Lambda} T^*M \twoheadrightarrow T^*M.$$
The anchor is given $\sharp(g,\alpha)=\Pi^\sharp(\alpha)$ and the bracket by:
$$[(f,\alpha),(g,\beta)]=(\Lie_{\Pi^\sharp(\alpha)}g-\Lie_{\Pi^\sharp(\beta)}f+\Lambda(\alpha,\beta),[\alpha,\beta]_{T^*M}).$$
Then $\partial_2:\pi_2(T^*M)\to M\times \mathbb{R}$ can be seen as an integration map associated to $\Lambda$:
\begin{equation}\label{integral-formula}\partial_2[S]=\int_{\Lambda} S.
 \end{equation}

More precisely, if $S=a_1\d t_1+a_2\d t_2:TI^2\to T^*M$,
 then $\partial_2[S]=\int_{I^2}\langle\Lambda,  a_1\wedge a_2\rangle$. Recall also that, once fixed a leaf $L$,
 one can always work with the subgroup bundle  $\pi_2(T^*L)$, rather than $\pi_2(T^*M)$.

In order to  understand \eqref{integral-formula} when $T^*M$ is integrable, recall that there is a multiplicative $2$-form $\omega_\Lambda$
 induced by $\Lambda$ on  $\G(T^*M)$. Moreover we know that $\pi_2(T^*M,x)$ identifies with $\pi_2(s^{-1}(x))$
 (the usual second fundamental group of the source-fiber above $x$). Then the  map $\partial_2$ comes as an integration
 of $\omega_\Lambda$ along the source fibers by the same argument as in \cite[Lemma 4.4]{CZ}.
 In the case $T^*M$ is not integrable, \eqref{integral-formula} still holds according to  \cite[Theorem 4.5]{Br}.
\end{ep}

\begin{ep}
One can slightly generalize the above situation by considering non-trivial representations:
let us consider a Lie algebroid $A$ and a representation $D$ of $A$ on a vector bundle $(E\to M,D)$ (equivalently,
 a flat $A$-connection on $E$).

Then for any $2$-cocyle $[\Lambda]\in H^2(A,E)$ represented by $\Lambda\in \Omega^2(A,E)$, we build a fibration as follows:
 $$E \hookrightarrow E\rtimes_\Lambda A \twoheadrightarrow A,$$
with anchor  $\sharp(g,\alpha)=\sharp_A(\alpha)$ and bracket:
$[(f,\alpha),(g,\beta)]=(D_{\alpha}g-D_{\beta}f+\Lambda(\alpha,\beta),[\alpha,\beta]_A)$.

In that case, one can still see the boundary map $\partial_2$ as an integration, but we have to use the parallel transport associated to $D$.
 Let $S=TI^2\to A$ be an $2$-sphere based at $x_0$ and covering $\gamma:I^2\to M$. Then according to
 \cite[Theorem 4.5]{Br}
$$\partial_2(S)=\int_{I^2}\Phi_{s}\langle \Lambda,S \rangle\  \d s,$$
where $\Phi_{s}:E_{\gamma(s)}\to E_{\gamma(0)}$ is obtained using the parallel transport along $a\circ \d i_{s_2}:TI\to A$, where $i_{s_2}(s_1):=(s_1,s_2)$ (in fact $s_1$ is varying from $t_1$ to
  $0$, and  $s_2=t_2$ is fixed).

For instance, if $A=TM$, then this means one can integrate spheres along (closed) $2$-forms with values in a flat vector bundle.
 In that case, the construction amounts to pull-back $E$ via $\gamma$, then trivialize $\gamma^*E$ using $D$, and then
 do the integration.
\end{ep}

\begin{ep}
Given a Poisson structure $\Pi\in\Gamma(\Lambda^2 TM)$ and a finite dimensional Lie algebra $\g$, a \emph{Poisson action up to homotopy} is defined \cite{Se}
as an abstract extension $\mathcal{E}$ of Lie algebras:
 $C^\infty(M)\hookrightarrow \mathcal{E}\twoheadrightarrow \mathfrak{g},$ where $C^\infty(M)$ is endowed with the Poisson bracket.
 Moreover, it is required that, for any $e\in\mathcal{E}$, the restriction of $\text{ad}_e$ to ${C^\infty(M)}$ is a derivation
 of the usual product of functions, \emph{i.e.} it is a vector field.

By choosing a splitting, we can identify $\mathcal{E}$ with $C^\infty(M)\oplus \g$, and the bracket necessarily takes the form:
\begin{eqnarray*}
\left[(0,\xi),(0,\xi')\right]_\mathcal{E}&=&(\omega(\xi,\xi'),[\xi,\xi']_\g)\\
\left[(f,0)(g,0)\right]_\mathcal{E}&=&(\{f,g\},0),\\
\left[(f,0)(0,\xi)\right]_\mathcal{E}&=&-\Lie_{D(\xi)}f,
\end{eqnarray*}
for some $\omega\in\Omega^2(\g,C^\infty(M))$ and $D:\g\to \Gamma(T^*M)$ satisfying:
\begin{eqnarray*}
 &\oint_{\xi_1,\xi_2,\xi_3} \Lie_{D(\xi_1)} \omega(\xi_2,\xi_3)-\omega([\xi_1,\xi_2],\xi_3)=0,&\\
 & D_{\left[\xi_1,\xi_2\right]}-\left[D_{\xi_1},D_{\xi_2}\right]=\Pi^\sharp(\d \omega(\xi_1,\xi_2)).
\end{eqnarray*}
Now since $T^*M$ has a natural Lie algebroid structure over $M$, one may try to build a Lie algebroid extension
 $T^*M\hookrightarrow \widehat{\mathcal{E}}\twoheadrightarrow \mathfrak{g}$. For this, we let $\widehat{\mathcal{E}}:=T^*M\oplus\g$ (note however that it is not clear how to pass
 directly from $\mathcal{E}$ to $\widehat{\mathcal{E}}$ without choosing a splitting, though this choice is clearly irrelevant up
 to isomorphism).

 Then  $\widehat{\mathcal{E}}$ is a Lie algebroid over $M$ with anchor
 $\sharp(\alpha,\xi)=\Pi^\sharp(\alpha)+D(\xi)$ and bracket  obtained by differentiating the above relations,
\begin{eqnarray*}
\left[(0,\xi),(0,\xi')\right]_{\widehat{\mathcal{E}}}&=&(\d\omega(\xi,\xi'),\left[\xi,\xi'\right]_\g)\\
\left[(\alpha,0)(\beta,0)\right]_{\widehat{\mathcal{E}}}&=&(\left[\alpha,\beta\right]_{T^*M},0),\\
\left[(\alpha,0)(0,\xi)\right]_{\widehat{\mathcal{E}}}&=&-\Lie_{D(\xi)}\alpha,
\end{eqnarray*}
for any $\alpha,\beta\in\Gamma(T^*M)$, $\xi_i\in\g$. It is a fibration if and only if $D(\xi)$ is a complete vector field
 $\forall\xi\in \g$ (for a good choice of a splitting). 

More generally, one may consider an abstract extension $A\hookrightarrow \widehat{\mathcal{E}}\twoheadrightarrow \mathfrak{g}.$
 In that case, there is a representation up to homotopy of $\g$ on the complex $\Omega^k(A)\xrightarrow{d_A}\Omega^{k+1}(A)$
 of $k$-forms on $A$ (this follows from  \cite[Section 3]{Br}).

\end{ep}

\section{Proof of the main theorem}
\subsection{Preliminaries}

The following proposition tells us how to build Lie algebroid morphisms out of time dependent sections:

\begin{prop}\label{prop:morphisms:sections}
Let $\alpha^t_1,\dots,\alpha^t_n$ be a family of sections of $A$, depending smoothly on a multi-parameter $t=(t_1,\dots ,t_n)\in I^n$
 and satisfying the following conditions:
\begin{equation}\label{commuting:t:sections}
[\alpha_i,\alpha_j]=\frac{\d \alpha_i}{\d t_j}-\frac{\d \alpha_j}{\d t_i},\quad\quad(i,j=1,\dots,n).
\end{equation}
Denote $X_i(x,t):=\sharp(\alpha_i^t(x))+\partial t_i$, and assume that $X_i$ is a complete vector field on $M\times I^n$
 for any $i\in\{1,\dots,n\}$. Then for any point $x_0$ in $M$, there exists $\gamma:I^n\to M$ satisfying:
\begin{eqnarray*}
 \frac{\d \gamma}{\d t_i}(t)&=&X^t_i(\gamma(t)),\\
 \gamma(0)&=&x_0.
\end{eqnarray*}
Moreover, $h:=\sum_{i=1}^n a_i \d t_i:TI^n\to A$ defines a Lie algebroid morphism, where:
 $$a_i(t):=\alpha_i^t(\gamma(t)).$$
\end{prop}
\begin{proof}
Let us first argue the existence of $\gamma$. Applying the anchor map to \eqref{commuting:t:sections}, we easily see that
 $X_i+\partial t_i$ for $i\in\{1,\dots,n\}$ mutually commute as vector fields on $M\times I^n$.
 By writing down the commutation relations of their respective flows on $M\times I^n$, one gets the following relations on the
 corresponding \emph{time-dependent} flows $\Phi_{s_i,s_i'}^{X_i}$ on $M$ (see  \cite[Prop. A.1]{Br} for more details when $n=2$):
$$\Phi_{s'_i,s_i}^{X_i^{(t_1,\dots,\hat{t}_i,\dots, s'_j,\dots,t_n)}}\circ\Phi_{s'_j,s_j}^{X_j^{(t_1,\dots,s_i,\dots,\hat{t}_j,\dots,t_n)}}=\Phi_{s'_j,s_j}^{X_j^{(t_1,\dots,s'_i,\dots,\hat{t}_j,\dots,t_n )}} \circ \Phi_{s'_i,s_i}^{X_i^{(t_1,\dots,\hat{t}_i,\dots,s_j,\dots,t_n)}}$$

Here, note that $X_i$ is a family of $t_i$-time dependent vector fields, \emph{depending on the parameters}
 $(t_1,\dots,\hat{t}_i,\dots,t_n)$, thus its time-dependent flow $\phi_{s_i,s_i'}^{X_i}$ also depends on
 $(t_1,\dots,\hat{t}_i,\dots, t_n)$. Using the above relations, it is easily checked by induction that the following formula
 provides a good candidate:
$$\gamma(t_1,\dots, t_n):=\Phi_{t_n,0}^{X_n^{(t_1,\dots, t_{n-1},\hat{t}_n)}}\circ \Phi_{t_{n-1},0}^{X_{n-1}^{(t_1,\dots, t_{n-2},\hat{t}_{n-1},0)}}\circ\dots\circ\Phi_{t_1,0}^{X_1^{(\hat{t}_1,0,\dots, 0)}}(x_0).$$

To show the second statement, recall that since $\d_A$ satisfies the Leibniz rule on $(\Omega^\bullet(A),\wedge)$,
 the condition for $h$ to define a Lie algebroid morphism, $\d_I\circ h^*=h^*\circ \d_A$, needs only to be checked on
 $\Omega^0(A)$ and $\Omega^1(A)$.

Moreover, it is a local condition, so we can choose a basis $e_1,\dots, e_k$ of sections of $A$, defined on a neighborhood of
 $\gamma(t_0)$ and work in local coordinates. We denote $e_1^*,\dots, e_k^*$ the dual basis, and  $c_{p,q}^l\in C^\infty(U)$
 the structure functions of $A$, they are defined by:
$$[e_p,e_q]=\sum_{l=1}^n c_{p,q}^{\, l} e_l,\quad\quad(p,q,l=1,\dots,k)$$ Locally, one can write
 $\alpha_i^t=\sum_p \alpha_{i,p}^t e_p$ so that $h=\sum a_{i,p}e_p\otimes \d t_i$ where
 $a_{i,p}=\alpha_{i,p}\circ \gamma$ $(i=1,\dots,n,p=1,\dots, k)$.

First, we compute that, for any $f\in \Omega^0(A)= C^\infty(M)$ and $i\in \{1,\dots, n\}$:
\begin{eqnarray*}
\langle h^*\circ \d_A(f),\partial t_i\rangle(t)&=&\langle \d f_{\gamma(t)},\sharp\circ h(\partial t_i)\rangle\\
                                               &=&\langle \d f,X_i^t(\gamma(t))\rangle\\
                                               &=&\frac{\d}{\d t_i}(f\circ\gamma)(t)\\
                                               &=&\langle \d h^*(f), \partial t_i \rangle(t).
\end{eqnarray*}
Thus, $h^*$ commutes with the differentials on $\Omega^0(A)$.

Then we express \eqref{commuting:t:sections} in local coordinates. We get that, for any $i,j\in\{1,\dots, n\}$ and $l\in\{1,\dots, k\}:$ 
\begin{eqnarray*}
\Lie_{X_i}\alpha_{j,l}-\Lie_{X_j}\alpha_{i,l}
+\frac{1}{2}\sum_{p,q=1}^n (c_{p,q}^{\, l}(\alpha_{i,p}\alpha_{j,q}-\alpha_{i,q}\alpha_{j,p})
         = \frac{\d \alpha_{i,l}}{\d t_j}-\frac{\d \alpha_{j,l}}{\d t_i}.
\end{eqnarray*}
By evaluating these equalities at the point $\gamma(t)$ and using the fact that $X_i^t(\gamma(t))=\frac{\d \gamma}{\d t_i}(t)$, we obtain the following equations:
\begin{eqnarray*}
\frac{1}{2}\sum_{p,q=1,\dots, k} c_{p,q}^{\, l} (a_{i,p} a_{j,q}-a_{i,q} a_{j,p})=\frac{\d a_{i,l}}{\d t_j}-\frac{\d a_{j,l}}{\d t_i}
\end{eqnarray*}
We can now check that:
\begin{eqnarray*}
\langle h^*\circ \d_A(e_l^*),\partial t_i\wedge \partial t_j\rangle
                 &=&\langle h^*(\sum_{p,q=1}^n c_{p,q}^{\, l} e_p^*\wedge e_q^*) , \partial t_i\wedge \partial t_j\rangle\\
                 &=&\langle \sum_{p,q=1}^n c_{p,q}^{\, l} e_p^*\wedge e_q^*, h(\partial t_i)\wedge h(\partial t_j)\rangle\\
                 &=&\frac{1}{2}\sum_{p,q=1,\dots, n} c_{i,j}^{\, l} (a_{i,p} a_{j,q}-a_{i,q} a_{j,p})\\
                 &=&\frac{\d a_{i,l}}{\d t_j}-\frac{\d a_{j,l}}{\d t_i}\\
                 &=&\langle \d \circ h^*(e_l^*),\partial t_i\wedge\partial t_j \rangle,\\
\end{eqnarray*}
which holds for any $l\in\{1,\dots, k\}$ and $i,j\in\{1,\dots, n\}$, so we conclude that $h^*$ commutes with differentials on $\Omega^1(A)$ as well.\end{proof}

Conversely, we would like to make sure that any morphism $a:TI^n\to A$ can be obtained as in
 Prop. \ref{prop:morphisms:sections}, for this we need to prove first:

\begin{prop} \label{prop:brackets}
 Let $\alpha^0_1,\dots,\alpha^0_n$ be a family of sections of $A$, depending smoothly on a multi-parameter
 $t=(t_1,\dots, t_n)\in I^n$ and satisfying the following conditions:
\begin{equation} 
[\alpha^0_i,\alpha^0_j]=\frac{\d \alpha^0_i}{\d t_j}-\frac{\d \alpha^0_j}{\d t_i},\quad\quad(i,j=1,\dots, n).
\end{equation}
Suppose that we are given a family of sections $\alpha_{n+1}$, depending on an extra parameter $(t_1,\dots,t_n,t_{n+1})$.

Then the unique solutions of the equations:
\begin{eqnarray}
 \label{prop:k:n+1} \frac{\d\alpha_k}{\d t_{n+1}}-\frac{\d\alpha_{n+1}}{\d t_k}&=&[\alpha_k,\alpha_{n+1}],\\
 \label{prop:k:n+1:ix} \alpha_k|_{\{t_{n+1}=0\}}&=&\alpha_k^0.
\end{eqnarray}
satisfy also the commuting relations:
\begin{equation}\label{prop:k:l}\frac{\d\alpha_k}{\d t_l}-\frac{\d\alpha_l}{\d t_k}=[\alpha_k,\alpha_l],
 \quad\quad (k,l=1,\dots, n+1)
\end{equation}
\end{prop}
\begin{proof} 
When $k\in\{1,\dots, n\}$ and $l\in\{n+1\}$, \eqref{prop:k:l} is just \eqref{prop:k:n+1}. For $k,l\in\{1,\dots, n\}$ one may
 consider the following expression:
$$\phi_{k,l}:=\frac{\d\alpha_k}{\d t_l}-\frac{\d\alpha_l}{\d t_k}-[\alpha_k,\alpha_l], \quad\quad (k,l=1,\dots, n). $$
A short computation shows that $\phi_{k,l}$ satisfies the equation:
$$\frac{\d \phi_{k,l}}{\d t_{n+1}}=[\alpha_{n+1},\phi_{k,l}].$$
The solution of such an equation is unique once given an initial condition $\phi_{k,l}|_{\{t_{n+1}=0\}}$.
 Thus if $\phi_{k,l}|_{\{t_{n+1}=0\}}$, then $\phi_{k,l}$ necessarily vanishes for any $t_{n+1}$. In other words, the relations
 \eqref{prop:k:l} for $k,l\in\{1,\dots, n\}$ only need to be checked for $t_{n+1}=0$, \emph{i.e.} for the $\alpha_k^0$'s.
\end{proof}

\begin{cor}\label{extend:morphisms}
 Let  $\sum a_i\d t_i:TI^n\to A $ be a Lie algebroid morphism. Then there exists a family of sections  $\alpha_1,\dots,\alpha_n$ such that:
\begin{eqnarray}
 \label{extend:morphisms:1} \frac{\d\alpha_k}{\d t_{l}}-\frac{\d\alpha_{l}}{\d t_k}&=&[\alpha_k,\alpha_{l}],\\
 \label{extend:morphism:2}  a_k&=&\alpha_k\circ \gamma
\end{eqnarray}
\end{cor}
\begin{proof}
The case $n=1$ is obvious, and \eqref{extend:morphisms:1} easily
follows from Proposition \ref{prop:brackets} by induction. The second statement \eqref{extend:morphism:2} was proved for the case
 $n=2$ in \cite[Prop. 1.3]{cf} by choosing an $A$-connection and arguing by uniqueness. But it is easy to apply the case $n=2$
 in each step of the induction (see below).

Let us now explicit the induction process when $n=3$:
\begin{enumerate}[i)]
 \item  we start with a Lie algebroid morphism $a_1\d t+a_2 \d t_2+a_3\d t_3:TI^3\to A$
 \item then we extend $a_1|_{\{t_2,t_3=0\}}$ into a time dependent section $\alpha_1$, up to now $\alpha_1$ only depends on $t_1$; in fact we are just fixing $\alpha_1|_{\{t_2,t_3=0\}}$.
 \item we extend arbitrarily $a_2|_{\{t_3=0\}}$ into a time-dependent section $\alpha_2$  (depending on $t_1,t_2$) and let $\alpha_1$ be the solution of the equation:
$$\frac{\d\alpha_1}{\d t_2}-\frac{\d \alpha_2}{\d t_1}=[\alpha_1,\alpha_2],$$
and initial conditions  $\alpha_1|_{\{t_2,t_3=0\}}$ given by step $ii)$. We get $\alpha_1$ depending on $t_1,t_2$. By \cite{cr}, we have necessarily $\alpha_1\circ\gamma|_{\{t_3=0\}}=a_1|_{\{t_3=0\}}$.
\item we extend $\alpha_3$ arbitrarily into a section depending on $t_1,t_2,t_3$ and consider the solutions:
$$\frac{\d\alpha_1}{\d t_3}-\frac{\d \alpha_3}{\d t_1}=[\alpha_1,\alpha_3],$$
$$\frac{\d\alpha_2}{\d t_3}-\frac{\d \alpha_3}{\d t_2}=[\alpha_2,\alpha_3],$$
with initial conditions on $\{t_3=0\}$ given by the last step. In order to make sure that $\alpha_1\circ \gamma=a_1$,
 we look at $a_1\d t_1+a_3\d t_3:TI^2\to A$, with $t_2$ as a fixed parameter. Similarly, we get $\alpha_2\circ\gamma=a_2$.
\end{enumerate}\end{proof}

\begin{remark}\label{rem:extend:spheres}
One observes from the previous reasonning that a Lie algebroid morphism $TI^n\to A$ is
entirely determined by (say) its restriction  to one face $d_{n,0}^n a$ and the values
$a_n(t)$ for $t\in I^n$.

Also, in the case one wants to extend an $A$-sphere $\sum a_k\d t_k:TI^n\to A$, it is not necessary to proceed by induction:
 one may extend first the last component $a_n$ into a time-dependent section $\alpha_n$, and then apply directly the
 Prop. \ref{prop:brackets} (with $n-1$ terms) and choose $\alpha_k|_{\{t_n=0\}}$, $k<n$ to all vanish. 
\end{remark}

\begin{lemma}\label{lemma:path-morphism}
 Let $a^{t_{n+1}}=\sum_{k=1}^n a(t_1,\dots,t_n,t_{n+1})\d t_k:TI^{n}\to A$ be a smooth family (parametrized by $t_{n+1}\in[0,1]$)
 of Lie algebroid morphisms that coincide on one face for all $t_{n+1}\in I$.

Then there exists a unique $a_{n+1}:I^{n+1}\to A$ such
 that $\sum_{i=1}^{n+1} a_k \d t_k:TI^{n+1}\to A$ is a Lie algebroid morphism.
\end{lemma}
\begin{proof}
 Say $a^{t_{n+1}}$ coincide on the face $\{t_n=0\}$, that is $a_k|_{\{t_n=0\}}$ is independent of $t_{n+1}$, $k=1,\dots,n$.
 We apply the procedure of Prop. \ref{extend:morphisms} with $n+1$ parameters and some careful choices:
\begin{itemize}
 \item first we extend $a_1|_{\{t_k=0, k=2,\dots,n\}}$ into a time dependent section $\alpha_1:I^2\to \Gamma(A)$ satisfying
  $\frac{\d \alpha_1}{\d t_{n+1}}=0$ for all $t_1,t_{n+1}\in I$;
 \item then we choose  $\alpha_2:I^2\to \Gamma(A)$ extending $a_2|_{\{t_k=0,k=3\dots n\}}$ and satisfying $\frac{\d \alpha_2}{\d t_{n+1}}=0$,
 $(\forall t_1,t_2,t_{n+1})\in I$. Then it is easily checked that the solution $\alpha_1$ of the equation:
$$\frac{\d\alpha_1}{\d t_2}-\frac{\d \alpha_2}{\d t_1}=[\alpha_1,\alpha_2],$$
with initial condition given by the previous step, satisfies:
$$\frac{\d}{\d t_2}\Bigl( \frac{\d\alpha_1}{\d t_{n+1}}\Bigr)=[\frac{\d\alpha_1}{\d t_{n+1}},\alpha_2].$$
Since $\frac{\d\alpha_1}{\d t_{n+1}}|_{\{t_2=0\}}=0$, and by uniqueness of the solution of the above equation, we conclude
 that $\frac{\d\alpha_1}{\d t_{n+1}}=0$ for all $t_1,t_2,t_{n+1}$;
\item we extend $a_3|_{\{t_k=0,k\geq4\}}$ into $\alpha_3$ such that $\frac{\d \alpha_3}{\d t_{n+1}}=0$, and so on...
\end{itemize}
We end up with $\alpha_1,\dots,\alpha_{n}$ depending on $t_1,\dots,t_{n+1}$ and satisfying
 $\frac{\d \alpha_k}{\d t_{n+1}}|_{\{t_n=0\}}=0$ for all $k=1,\dots,n.$
Then we consider  the solution $\alpha_{n+1}$ of the equation:
\begin{eqnarray*}
\frac{\d \alpha_{n+1}}{\d t_n}-\frac{\d \alpha_n}{\d t_{n+1}}&=& [\alpha_{n+1},\alpha_n],\\
\alpha_{n+1}|_{\{t_n=0\}} & = & 0.
\end{eqnarray*}
In order to apply Prop. \ref{prop:morphisms:sections}, we need to make sure that $\alpha_{n+1},\alpha_k$ satisfy the commutation
 relations.
 For this, we let:
$$\phi_{n+1,k}:=\frac{\d\alpha_{n+1}}{\d t_k}-\frac{\d \alpha_k}{\d t_{n+1}}-[\alpha_{n+1},\alpha_k],$$
and check that $\phi_{n+1,k}$ satisfies:
$$\frac{\d \phi_{n+1,k}}{\d t_n}=[\alpha_n,\phi_{n+1,k}].$$
Since $\phi_{n+1,k}|_{\{t_n=0\}}=0$ (by our choice of $\alpha_k$) we conclude that $\phi_{n+1,k}=0$ for all $t_n \in I$.
 We will leave the proof of uniqueness to the interested reader.
\end{proof}

\begin{thm}\label{thm:lift}
Given a complete Ehresmann connection $\sigma$ of $\pi: A_E \to A_B$,  there is  a canonical lift
 $\ta: TI^{n} \to A_E$ in the following
commutative diagram of Lie algebroid morphisms:
\begin{equation}\label{diag:lift-a}
\xymatrix{
TI^{n-1} \ar[r]^{\ta^0} \ar[d]_{\d i_{n,0}^n} & A_E \ar[d]_{\pi} \\
TI^{n} \ar[r]^{a} \ar[ur]^{\exists  \ta} & A_B 
 }
\end{equation}
Moreover, if $\tilde{a},\tilde{a}'$ are the lifts obtained by two different Ehresmann connections $\sigma$ and $\sigma'$, then
 there is a Lie algebroid morphism $h: TI^{n+1} \to A_E$ linking $\ta$ and $ \ta'$, that is, satisfying:
\begin{equation*}
 \quad  d_{n+1,0}^{n+1}h=\ta, \quad\quad d_{n+1,1}^{n+1}h=\ta'. 
\end{equation*}\end{thm}
\begin{proof}
By Cor. \ref{extend:morphisms}, we extend  $a: TI^n \to A_B$ and
$\tilde{a}^0: TI^{n-1} \to A_E$ into time dependent sections $\alpha_1, \dots,
\alpha_n$ and $\talpha^0_1, \dots, \talpha^0_{n-1}$ respectively (in fact $\tilde{\alpha_i}^0$ can be chosen to $\pi$-project
onto $\alpha_i|_{\{t_n=0\}}$ but this is not relevant for our matters here).

We lift $\alpha_n$ to $\talpha_n:=\sigma(\alpha_n)$ using the
Ehresmann connection $\sigma$, then we let $\talpha_i$ for $i=1,
\dots, n-1$  be the solutions of the differential equations
\begin{equation}\label{eq:lift}
\begin{split}
\frac{\d\talpha_i}{\d t_{n}} - \frac{\d\talpha_{n}}{\d t_i}& = [\talpha_i,\talpha_n] \\
\talpha_i|_{\{t_n=0\}} & = \talpha^0_i
\end{split}
\end{equation}
By Prop. \ref{prop:brackets} and Prop. \ref{prop:morphisms:sections}, we obtain a Lie algebroid morphism
 $TI^n \xrightarrow{\tilde{a}} A_E$ such that the base morphism $I^n \xrightarrow{\tg} E$ satisfies  $\tg \circ i_{n,0}^n=\tg^0$.
 Here $\tg^0$ is the base morphism of $\tilde{a}^0$. In fact $\tg$ is the unique lift of the base map $\gamma$ of $a$ such that
  $\tg \circ i_{n,0}^n=\tg^0$ and $\partial_{t_n} \tg$ lies in the image of the horizontal part $H$ of the Ehresmann connection $\sigma$ under the anchor map.
 By construction, $\talpha|_{\{t_n=0\}} = \talpha^0$ so $\tilde{a}^0=i^n_{n,0}\tilde{a}$, and $\tilde{a}$ projects onto
 $a$ via $\pi$.

Then $\ta$ is canonical in the sense that it does not depend on the choice of the extensions $\alpha_i$'s and $\talpha^0_i$'s.
 This again can be argued by choosing an $A$-connection and reduce \eqref{eq:lift} to an equation with functions $\ta_i$. See also Remark \ref{rem:extend:spheres}.

Assume now that we have two Ehresmann connections $\sigma$ and $\sigma'$. Since the space of Ehresmann connections is an affine space based on the vector space
 $\Gamma(A^*_B)\otimes \Gamma(\cK)$, the path $\sigma_s=(1-s)\sigma+s\sigma'$ connects
 both connections. Therefore, we can apply the above construction with parameter $s$ by letting $\tilde{\alpha}_n=\sigma_s(\alpha_n)$. 
 Then we obtain a path of lifts $\ta(s): TI^n \to A_E$ connecting the two lifts via $\sigma$ and $\sigma'$ respectively.
 By Lemma \ref{lemma:path-morphism}, we can build a Lie algebroid morphism $h:TI^{n+1} \to A_E$ linking them by extending this path.
\end{proof}

\subsection{The transgression map}\label{transgression}

\subsubsection{Construction of the transgression map $\partial_n: \pi_n(A_B) \to \pi_{n-1} (\cK)$} \label{sec:const-trans}
Recall that  an element in $\pi_n(A_B)$ is represented by a sphere, namely a Lie algebroid morphism 
\[\text{$a=\sum_{k=1}^n a_k d t_k: TI^{n} \to A_B$ with $a_k(t_1, \dots, t_n)=0$ if one of $\{t_1, \dots, \hat{t}_k, \dots, t_n\}$ is $0$ or 1.}\] 
We denote $b_0\in B$ the point $a$ is based at, and for any $x_0\in E_{b_0}$, we apply
 the construction in the proof of Theorem \ref{thm:lift} with the following choices:
\begin{itemize}
\item $\ta^0=0_{x_0}$;
\item the time dependent sections $\talpha^0_i=0$, for all $i=0,\dots,n-1$;
\item $\alpha_n$ satisfying $\alpha_n|_{\{t_k=0\; \text{or} \; 1\}} = 0$, for all $k=1,\dots, n-1$.
\end{itemize}
We obtain a lift $\ta: TI^n \to A_E$ such that $d^n_{n,0}\ta=0$. Note that $\ta$ is not necessary a sphere anymore.
 However, it is zero on the boundary of $I^n$ except the interior of the face $\{t_n=1\}$. The reason is the following:
 when we restrict the equation \eqref{eq:lift} on $\{t_k=0,1\}$ for $k=1,\dots, n-1$, we have $\talpha_n|_{\{t_k=0, 1\}} = 0$
  since $\alpha_n$ has this property, so it is easily seen that the solution $\talpha_i$ with initial
 value $\talpha^0_i=0$ identically vanishes, that is $\talpha_i|_{\{t_k=0,1\}}=0$. Hence, we obtain: 
\[\text{ $\ta_i|_{\{t_k=0, 1\}}=0$ for any $k=1, \dots, n-1$ and $i=1, \dots, n$. }\]
Therefore $d_{n,1}^n\ta=\sum_{k=1}^{n-1} \ta_k|_{\{t_n=1\}} dt_k : TI^{n-1} \to 
A_E $ satisfies the correct boundary conditions to make it a $(n-1)$-sphere.
Now we may assume that the sphere we choose to represent satisfies $d^n_{n, 1}a=0$ by Remark \ref{rem:squares:vs:spheres} from the very beginning.  Then $d^n_{n,1} \ta$ is a sphere in $\cK$ since $\pi(d^n_{n,1}\ta)=d^n_{n,1}a=0$.

 We thus define:
 $$\partial_n([a])_{x_0}:=[\d_{n,1}^n\ta]\in\pi_{n-1}(\K,x_0).$$

\subsubsection{Well-definedness of $\partial_n$.}\label{sec:well-define}
The argument below will be similar as in Lemma \ref{lemma:path-morphism}, but we need to be more precise on the homotopy $h$,
 so let us give a detailed explanation.

Let $h$ be a homotopy between $a^-=d_{n+1,0}^{n+1}h$ and $a^+=d_{n+1,1}^{n+1}h$. For simplicity, we still assume $d^n_{n, 1}a^+=d^n_{n, 1}a^-=d^n_{n, 1}h=0$ by Remark \ref{rem:squares:vs:spheres}. In particular we have a family of
 spheres $a^{t_{n+1}}$ based at a same point so we can do the construction of the transgression map with parameter $t_{n+1}$
 (we keep the same notations as in Section \ref{sec:const-trans} but with an extra parameter $t_{n+1}$).

Then we consider the solution $\talpha_{n+1}$ of the equation:
\begin{equation}\label{eq:well:def}
\begin{split}
\frac{\d \talpha_{n+1}}{\d t_{n}}-\frac{\d \sigma(\alpha_{n})}{\d t_{n+1}}& =  [\talpha_{n+1},\sigma(\alpha_n)] \\
\talpha_{n+1}|_{\{t_{n}=0\}} & = 0,
\end{split}
\end{equation}
For $k=1, \dots, n-1$,  $\talpha_{n+1}|_{\{t_k=0,1\}}=0$ since we chose $\alpha_n$ vanishing on this set for all $t_{n+1}$.
 Moreover, if we let $\phi_{k,n+1}:=\frac{\d \talpha_{n+1}}{t_{k}}-\frac{\d \talpha_{k}}{\d t_{n+1}}- [\talpha_{n+1},\talpha_k]$,
 a straightforward computation shows that $\phi_{k,n+1}$ satisfies:
 $$\frac{\d \phi_{k,n+1}}{\d t_n}=[\sigma(\alpha_n),\phi_{k,n+1}],$$
 with initial condition $\phi_{k,n+1}=0$. By uniqueness of the solution of such an equation, we see that $\phi_{k,n+1}=0$.
We conclude that $\frac{\d \talpha_{n+1}}{\d t_{k}}-\frac{\d \talpha_{k}}{\d t_{n+1}}=[\talpha_{n+1},\talpha_k]$
and by applying Prop. \ref{prop:morphisms:sections} and Prop. \ref{prop:brackets},  we obtain a Lie algebroid morphism $\tilde{h}:TI^{n+1}\to A$, with boundary property
\begin{equation} \label{eq:h-boundary}
\text{$\tilde{h}_{n+1}|_{\{ t_k=0\; \text{or} \; 1\}} = 0$ for $k=1, \dots, n-1$.}\end{equation}
Thus the face $\{t_{n}=1\}$ of $\tilde{h}$ induces a $\K$-homotopy between $\d_{n,1}^{n}\tilde{a}^-$ and $d_{n,1}^n\tilde{a}^+$ since $\pi(\tilde{h}|_{t_n=1})=h|_{t_n=1}=0$.
 
A similar argument applies to show that the transgression map is independent of the connection: assume we have two connections
 $\sigma_0,\sigma_1$. Then we define a family of connections by $\sigma_{t_n+1}=(1-t_{n+1})\sigma_0+t_{n+1}\sigma_1$.
 Then one can use $\sigma_{t_{n+1}}$ to do the construction of \ref{sec:const-trans} with parameter $t_{n+1}$.
 Then the above construction holds, yielding a $\K$-homotopy  between $\d_{n,1}^{n}\tilde{a}^0$ and $\d_{n,1}^n \tilde{a}^1$,
 where $\tilde{a}^i$ is the lift of $a$ obtained using $\sigma_i$.

\subsubsection{$\partial_n$ is a group morphism}\label{sec:group-morph}
Given $a^+, a^- \in S_n(A_B)$, then $[a^+]\cdot [a^-]$ can be represented by $a^+ \odot_{t_{n-1}} a^-$ by Lemma \ref{lemma:diff-glue}. Then we have the following commutative diagram
\[
\xymatrix{ TI^{n-1} \sqcup TI^{n-1} \ar[r]^{0} \ar[d]_{t_n=0} & A_E \ar[d]^{\pi} \\
TI^n\sqcup TI^n \ar[r]^{a^+ \sqcup a^-} \ar[ur]^{\ta^+ \sqcup \ta^-} & A_B 
}
\]
Since concatenation with respect to $t_{n-1}$ is the composed map $TI^n \xrightarrow{g_{t_{n-1}}}  TI^n \sqcup TI^n \xrightarrow{a^+\sqcup a^-} A_B$, it is not hard to see that the lift of $a^+ \odot_{t_{n-1}} a^-$ restricted to $t_n=1$ satisfies
\[\widetilde{a^+ \odot_{t_{n-1}} a^-}|_{t_n=1} = \left( \ta^+ \odot_{t_{n-1}} \ta^- \right) |_{t_n=1}  =\ta^+ |_{t_n=1} \odot_{t_{n-1}} \ta^- |_{t_n=1}.\]
Thus $\partial_n([a^+]\cdot [a^-])=\partial_n([a^+])\cdot \partial_n([a^-])$.  

\subsection{Exactness}
\begin{prop}\label{lem:sE:decomposition}
 Let $\K\hookrightarrow A_E\twoheadrightarrow A_B$ be a fibration. Then any sphere $S\in \mathcal{S}_n(A_E)$ is homotopic to one of the form: 
 $$S\simeq C_\K \odot C_H, $$
 where $C_\K:TI^n\to \K$ and $C_H:TI^n\to A_E$ are concatenated along a $(n-1)$-sphere representing $\partial_n(\pi(S))\in{\pi_{n-1}(\K)}.$
\end{prop}

\begin{remark}
In Proposition \ref{lem:sE:decomposition}, neither $C_\K$ nor $C_H$ is supposed to be a sphere.
 However, since $C_\K$ lies in the kernel of $\pi$, we see that $C_H$ and $S$ have same image by $\pi$.

As we shall see, $C_H$ is a lift of the sphere $\pi(S)$. It might not be a sphere itself,  however,
 the failure of $C_H$ to be a sphere is measured by its boundary: it is a $(n-1)$-sphere whose homotopy class is precisely
 $\partial([\pi(S)])$. \\
\psfrag{x}{$x$} \psfrag{ck}{$C_\K$} \psfrag{ch}{$C_H$} \psfrag{homotopic to}{homotopic to} \psfrag{pi}{$\downarrow \pi$}
\centerline{\epsfig{file=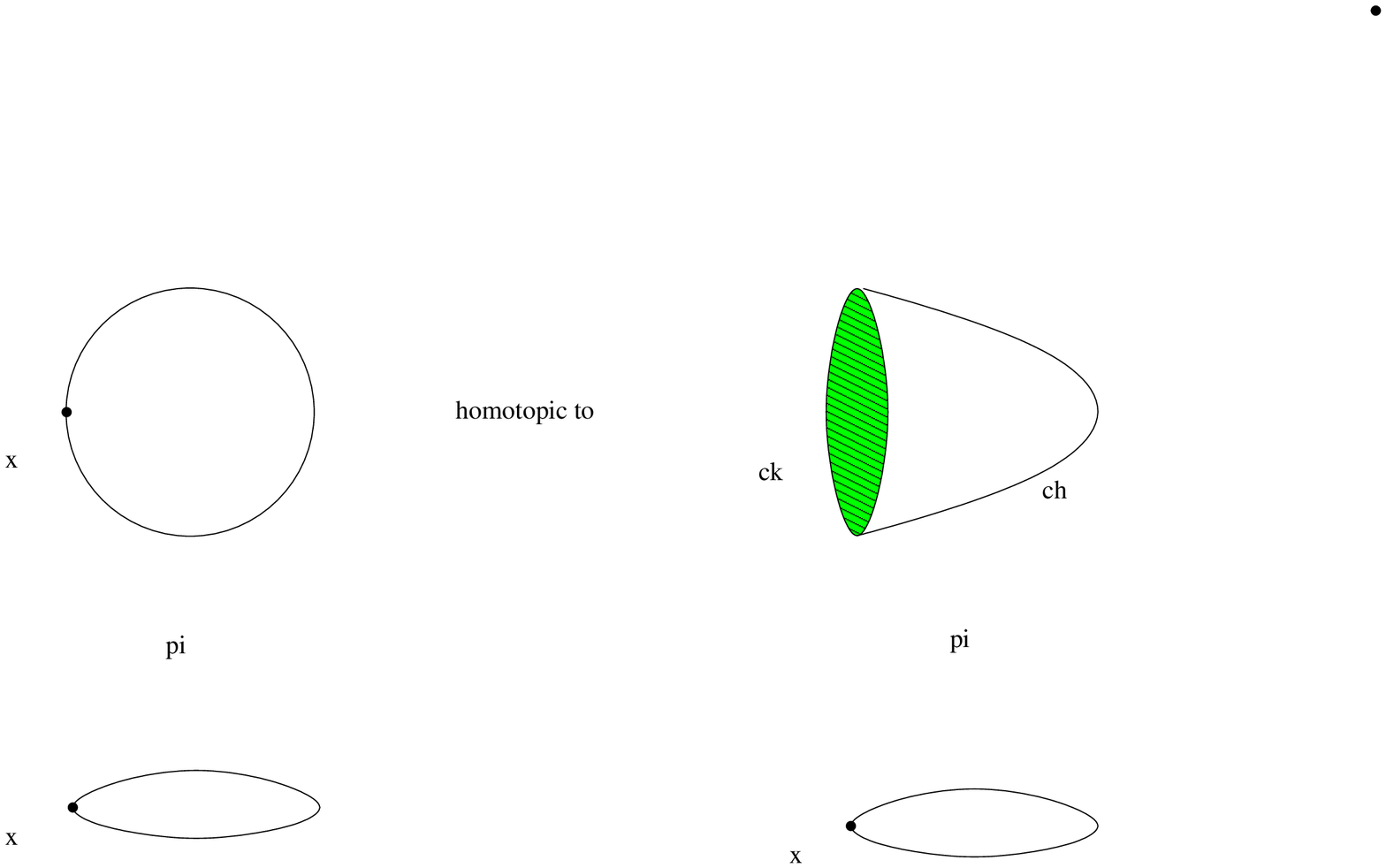, height=4.5cm, width=12cm} }
\end{remark}
\begin{remark}
 In the case $n=1$, the proof below will show that given a connection, any $A_E$-path is homotopic
 to the concatenation of a $\K$-path with a horizontal path (\emph{i.e} path in $\im \sigma$).
 In order to understand this proof, it is actually a good idea to focus on that case.
\end{remark}

\begin{proof}
 We pick a connection $\sigma$, denote $S=\sum_{k=1}^n a_k \d t_k$ and extend $a_n:I^n\to A_E$ into a time-dependent section
 of $A_E$ of the form $\alpha_n^H+\alpha_n^\K$ as follows:
\begin{Itemize}
 \item $\alpha_n^H=\sigma(\alpha_n^B)$, where $\alpha_n^B$ is any time-dependent section of $A_B$ extending $\pi\circ a_n$,
 \item $\alpha_n^\K$ is any time-dependent section of $\K$ such that $\alpha_n^H+\alpha_n^\K$ extends $a_n$,
 \item  $\alpha_n^B$ and $\alpha_n^\K$ are chosen so that they vanish whenever one of $\{t_1,\dots, t_n\}$ is $0$ or $1$.
\end{Itemize}
Then we introduce an extra variable $t_{n+1}$ and define a family of sections $\alpha_{n+1}$ depending on
 $(t_1,\dots,t_{n+1})\in I^{n+1}$ and defined by the formula:
\begin{equation}\label{eq:split}
\alpha_{n+1}(t_1,\dots,t_{n+1}):=-t_n. \alpha_n^H(t_1,\dots,t_{n-1},t_n(1-t_{n+1})).
\end{equation}
 The above left hand side will be written $-t_n\alpha_n^H(t_n(1-t_{n+1}))$ in order to keep notations short.
 Note  that $\alpha_{n+1}$ is chosen such that, as a family of $t_{n+1}$-time dependent section with parameters $(t_1,\dots,t_n)$,
 its flow satisfies: $\Phi^{\alpha_{n+1}}_{1,0}=\Phi_{t_n,1}^{\alpha^n_H}$ (here, $\Phi^\alpha$ denotes the flow of the linear
 vector field on $A$ induced by $[\alpha,-]$, see \cite[A.1]{cf}). In fact, we will use $\alpha_{n+1}$ to transport $S$
 into a morphism with values in $\K$ (as we shall see, its flow will somehow cancel the $\alpha^H_n$ component of $a_n$).

Now we set $\alpha_n^0:=\alpha_n^H+\alpha^K_n$, and let $\alpha^0_k$  $(k=1,\dots,n-1)$ be the solution of:
\begin{eqnarray}
\label{prop:sE:n:k}   \frac{\d \alpha_k^0}{\d t_{n}}-\frac{\d \alpha^0_{n}}{\d t_{k}}&=&[\alpha_k^0,\alpha^0_{n}],\\
\label{prop:sE:n0}                       \alpha_k^0|_{\{t_n=0\}}&=&0.
\end{eqnarray}
As explained in Remark \ref{rem:extend:spheres}, $\alpha_k^0$ satisfy the commutation relations, and extend the morphism $S$.
 Then we let $\alpha_k$ $(k=1,\dots,n)$ be the solution of the equation:
\begin{eqnarray}
 \label{prop:sE:k:n+1} \frac{\d\alpha_k}{\d t_{n+1}}-\frac{\d\alpha_{n+1}}{\d t_k}&=&[\alpha_k,\alpha_{n+1}],\\
 \label{prop:sE:k:n+1:ic} \alpha_k|_{\{t_{n+1}=0\}}&=&\alpha_k^0.
\end{eqnarray}
By applying Proposition \ref{prop:brackets} we see that $\alpha_k$'s satisfy the commutation relations:
\begin{equation}\label{prop:sE:k:l}\frac{\d\alpha_k}{\d t_l}-\frac{\d\alpha_l}{\d t_k}=[\alpha_k,\alpha_l], \quad\quad (k,l=1,\dots,n+1)\end{equation}

 Thus Proposition \ref{prop:morphisms:sections} applies (with $n+1$ variables) so that one gets a Lie algebroid morphism  $\lambda:TI^{n+1}\to A_E$ by setting $\lambda:=\sum a_k \d t_k,$ where $a^t_k:=\alpha^t_k(\gamma(t))$ with $ (k=1,\dots,n+1).$

Let us focus on $\alpha_n$.

\begin{claim}\label{lemma:sE:claim1} ${\alpha_n}|_{\{t_{n+1}=1\}}$ has values in $\Gamma(\K)$.
\end{claim}

\begin{proof}[Proof of Claim \ref{lemma:sE:claim1}] The following is an  explicit formula for the solution $\alpha_n$ of
 \eqref{prop:sE:k:n+1}:
\begin{equation}\label{eq:explicit:alpha:n}
 \alpha_n=
      \int_0^{t_{n+1}} \bigl(\Phi^{\alpha_{n+1}}_{t_{n+1},s}\bigr)_*\Bigl({\frac{\d\alpha_{n+1}}{\d t_n}}|_{\{t_{n+1}=s\}}\Bigr)\d s
     +(\Phi^{\alpha_{n+1}}_{t_{n+1},0})_*\bigl(\alpha_n|_{\{t_{n+1}=0\}}\bigr).
\end{equation}
This can be checked by a direct computation (in fact it is the same as in \cite[Prop 1.3]{cf} with an extra term coming
 from the initial condition).

In the first term of this expression, we first compute from \eqref{eq:split} that:
\begin{equation}
\frac{\d\alpha_{n+1}}{\d t_{n}}=\frac{\d}{\d t_{n+1}} \Bigl((1-t_{n+1})\alpha^H_n(t_n(1-t_{n+1}))\Bigr).
\end{equation}
Then we recall the Cartan formula for Lie algebroids: given a $t$-time dependent section $\mu$, and $\nu$ a smooth family of
 sections depending on $t$, we have:
\begin{equation*}
 \bigl(\Phi_{t,s}^{\mu}\bigr)_*\Bigl(\frac{\d \nu}{\d t}|_{\{t=s\}}+\bigl[\mu|_{\{t=s\}},\nu|_{\{t= s\}}\bigr]\Bigr)=\frac{\d}{\d s}(\Phi_{t,s}^{\mu})_*(\nu).
\end{equation*}
Applied with $t=t_{n+1}$, $\mu=\alpha_{n+1}$ and $\nu=(1-t_{n+1})\alpha^H_n(t_n(1-t_{n+1}))$
 (so that $\bigl[\mu|_{\{t=s\}},\nu|_{\{t= s\}}\bigr]=0$) we obtain the following expression for the first term of
 \eqref{eq:explicit:alpha:n}:
\begin{eqnarray*}
\int_0^{t_{n+1}} \bigl(\Phi^{\alpha_{n+1}}_{t_{n+1},s}\bigr)_*\Bigl({\frac{\d\alpha_{n+1}}{\d t_n}}|_{\{t_{n+1}=s\}}\Bigr)\d s
&=&\int_0^{t_{n+1}} \Bigl[\frac{\d}{\d s}\, (\Phi^{\alpha_{n+1}}_{t_{n+1},s})_*\bigl((1-s)\alpha^H_n(t_n(1-s))\bigr) \Bigr]\d s\\
&=&(1-t_{n+1})\alpha^H_n(t_n(1-t_{n+1}))-(\Phi^{\alpha_{n+1}}_{t_{n+1},0})_*\bigl(\alpha^H_n(t_n)\bigr).
\end{eqnarray*}
In particular, when setting $t_{n+1}=1$ and taking into account the second term in \eqref{eq:explicit:alpha:n} we get:
\begin{eqnarray*}
\alpha_n|_{\{t_{n+1}=1\}}
&=&-(\Phi^{\alpha_{n+1}}_{1,0})_*\bigl(\alpha^H_n(t_n)\bigr)+(\Phi^{\alpha_{n+1}}_{1,0})_*\bigl( \alpha_n^H+\alpha_n^\K   \bigr)\\
&=&(\Phi^{\alpha_{n+1}}_{1,0})_*\bigl( \alpha_n^\K   \bigr)\\
&=&(\Phi^{\alpha^H_{n}}_{t_n,1})_*(\alpha_n^\K).
\end{eqnarray*}
Since $\alpha^H_n$ is $\pi$-projectable, its flow preserves sections of $\K$ (this is an easy consequence of Lemma 1.9 in \cite{Br}).
 Thus $\alpha_n|_{\{t_{n+1}=1\}}$ is indeed a section of $\K$, as claimed.
\end{proof}

We now list the non-trivial faces of $\lambda$.
\begin{claim}\label{lemma:sE:claim3} $\lambda:TI^{n+1}\to A_E$ satisfies:
\begin{Enumeratei}
\item $d_{n+1,0}^{n+1} \lambda=S$
\item $d_{n+1,1}^{n+1} \lambda=:C_\K$ has values in $\K$;
\item \hspace{7pt}$d_{n,1}^{n+1} \lambda=:C_H$ lifts  $\pi(S)$;
\end{Enumeratei}

\end{claim}
\begin{proof}[Proof of claim \ref{lemma:sE:claim3}]\ 

$i)$ follows from \eqref{prop:sE:k:n+1:ic} and by applying the Remark \ref{rem:extend:spheres} to \eqref{prop:sE:n:k} and
 \eqref{prop:sE:n0};

$ii)$ is a consequence of the Claim \ref{lemma:sE:claim1}. More precisely, since $\alpha_n$ has values in $\Gamma(\K)$ on
 $\{t_{n+1}=1\}$, when setting  $(l=n,\ t_{n+1}=1)$ in \eqref{prop:sE:k:l} for all  $k=1,\dots, n-1$, we see that
 $\alpha_k|_{\{t_{n+1}=1\}}$ has values in $\Gamma(\K)$ provided this is true on $\{t_n=0,\ t_{n+1}=1\}$.
 So it is enough to show that $\alpha_k=0$ on ${\{t_n=0,t_{n+1}=1\}}$.

To show this, we set $(t_n=0)$ in \eqref{prop:sE:k:n+1}: since $\alpha_{n+1}$  identically vanishes on ${\{t_n=0\}}$
 (by  \eqref{eq:split}), we conclude that $\alpha_k|_{\{t_n=0,t_{n+1}=1\}}$ vanishes provided it does at $(t_{n+1}=0)$.
 But  $\alpha_k|_{\{t_{n+1}=0,t_n=0\}}=0$ is just \eqref{prop:sE:n0}.

$iii)$ follows from the formula \eqref{eq:split} for $\alpha_{n+1}$ since $\alpha_{n+1}|_{\{t_{n}=1\}}$ coincides with $\alpha_n^H$,
 and by our construction of the transgression map $\partial_n$.\end{proof}

\begin{claim}\label{lemma:sE:claim4}
 All the other faces of $\lambda$ vanish.
\end{claim}
\begin{proof}[Proof of \ref{lemma:sE:claim4}]
First, $d_{n,0}^{n+1} \lambda=0$ easily follows from \eqref{eq:split}. Then the remaining faces are $d_{k,\epsilon}\lambda$ for any
 $k<n$ and $\epsilon\in\{0,1\}$. Moreover, $d_{k,\epsilon}^{n+1}\lambda=0$ is by definition equivalent to the following statements:
\vspace{2pt}

\hspace{-5pt}\vspace{2pt}\begin{tabular}{crcl}
 $a)$& $a_{n+1}|_{\{t_k=\epsilon\}}=$&\!\!\!$0,$\vspace{2pt}\\
 $b)$& $a_{n}|_{\{t_k=\epsilon\}}=$&\!\!\!$0,$\vspace{2pt}\\
 $c)$& $a_{l}|_{\{t_k=\epsilon\}}=$&\!\!\!$0,$& for any $l<n,l\neq k$\vspace{2pt}.
\end{tabular}

\noindent \vspace{2pt}Here, we can see that:

$a)$ is clear because of \eqref{eq:split} and by our choices on $\alpha^H_n+\alpha_n^\K$ to vanish when $t_k=0,1$.

$b)$ is true when $(t_{n+1}=0)$ by our choice of $\alpha_n|_{\{t_k=\epsilon,t_{n+1}=0\}}=0$. On the other hand, setting
 $t_k=\epsilon$ in \eqref{prop:sE:k:n+1}, since $\alpha_{n+1}|_{\{t_k=1\}} \forall t_{n+1}$, then necessarily
 $a_{n}|_{\{t_k=\epsilon\}}=0$ for all $t_{n+1}$.

$c)$ may argued as follows: when $t_{n+1}=0$, we have $\alpha_l|_{\{t_k=\epsilon,t_{n+1}=0\}}=\alpha_l^0|_{\{t_k=\epsilon\}}$.
 If one can prove that $\alpha_l^0|_{\{t_k=\epsilon\}}=0$, then necessarily $a_{l}|_{\{t_k=\epsilon\}}=0$ for all $t_{n+1}$;
 this follows by setting ${t_k=0}$ in \eqref{prop:sE:k:n+1}, then using $a)$.

Now to see that $\alpha_l^0|_{\{t_k=0\}}$, one just sets $t_k=0$ in \eqref{prop:sE:n:k}, then uses $b)$.\end{proof}

Thus, we see that $\lambda$ induces a homotopy between $d_{n+1,0}^{n+1} \lambda=S$ and the concatenations of its two other non
 trivial faces $d_{n,1}^{n+1} \lambda \odot_n d_{n+1,0}^{n+1}\lambda$. One can be more explicit by considering $\lambda\circ \d h$
 for some well-chosen smooth map $h:I^{n+1}\to I^{n+1}$.\end{proof}

\begin{remark}
 In the case $n=1$, Proposition \ref{lem:sE:decomposition} says that any $A$-path is homotopic to the concatenation of a $\K$-path with a $A_E$-path with values in $H$.
\end{remark}

\begin{cor} 
 The long sequence \eqref{long:sequence1} is exact at $\pi_n(A_B)$.
\end{cor}
\begin{proof}
First we show $\ker \partial_n \subset \im \pi_{n}$. For this consider $S_B \in \mathcal{S}_n(A_B)$ such that $\partial_n[S_B]=0$. We take the lift $\tilde{S}_B$ of  $S_B$ as in Section  \ref{sec:const-trans}. Since $\partial_n[S_B]=0$,  $\tilde{S}_B|_{\{t_n=1\}}$  is a contractile sphere in $\K$. Thus there exists $C_\K: TI^n\to \K$ such that the concatenation $C_K \odot\tilde{S}_B$ along $\tilde{S}_B|_{t_n=1}$ is a $n$-sphere in $A_E$. Thus  $S_B=\pi(C_K \odot\tilde{S}_B)$ represents an element in $\im \pi_n$. 

Reciprocally, we have $\im \pi_{n}\subset \ker \partial_{n}$. Consider $[S_B]\in \im\pi_{n}$, that is  $S_B=\pi(S)$ for some $S\in \mathcal{S}_n(A_E)$. Then by the Proposition \ref{lem:sE:decomposition}, we can replace $S$ by $C_H\odot C_\K$, concatenated along a $(n-1)$-sphere $S_K$ that represents $\partial_{n} [S_B]$. Since $S_K$ bounds $C_\K$, $S_K$ is contractible. Hence  $\partial_{n}[S_B]=0$.
\end{proof}

\begin{cor}\label{cor:A_E}
The long sequence \eqref{long:sequence1} is exact at $\pi_n(A_E)$.
\end{cor}
\begin{proof}
It is clear that $\im i_n\subset \ker \pi_n$. We now verify that $\ker \pi_n \subset \im i_n$. Take $S_E \in \mathcal{S}_n(A_E)$
such that $\pi(S_E)=:S_B$ is homotopic to 0 via a certain $h$. We take the lift $\tilde{S}_B$ of $S_B$ as in Section \ref{sec:const-trans}, then by
 Prop. \ref{lem:sE:decomposition}, $S_E$ is homotopic to the concatenation $C_\cK \odot \tilde{S}_B$ along
 $S_\cK\in \mathcal{S}_{n-1} (\cK)$, which represents $\partial_n([S_B])=0$. Recalling Section \ref{sec:well-define}, $S_\cK$ is homotopic to 0 via $\tilde{h}|_{t_n=1}$, where $\tilde{h}$ is a well-chosen lift of $h$. 
 Thus the concatenation $C_\cK \odot \tilde{h}|_{t_n=1}$ along $S_\cK$ is a sphere in $\cK$ which is homotopic to
 $C_\cK \odot \tilde{S}_B$ via $id\odot \tilde{h}$ by \eqref{eq:h-boundary}. Thus $[S_E]=   [C_\cK \odot \tilde{h}|_{t_n=1}]$ is in $\im i_n$.
\end{proof}

\begin{cor}
 The sequence \eqref{long:sequence1} is exact at $\pi_{n-1}(\K)$.
\end{cor}
\begin{proof}
It is clear that $\im \partial_{n}\subset \ker i_{n-1}$. In order to prove the inverse inclusion, we need to show that any $\K$-sphere
 $S_\K$ which is contractile in $A_E$, is $\K$-homotopic to one of the form $\partial_{n}(S_B)_{x_0}$,
 for some sphere $S_B\in \mathcal{S}_{n}(A_B)$ (to be defined). Here, the notation $\partial_{n}(S_B)_{x_0}$
 stands for $d_{n,1}^n\tilde{a}$, with $\tilde{a}$ as in \ref{sec:const-trans}, and $x_0$ is the point $S_\K$ is based at. 

To do this, we will adapt the proof of Proposition \ref{lem:sE:decomposition}: we replace $S$ by $H:TI^n\to A_E$, an $A_E$-homotopy
 between $S_\K=d^{n}_{n,0}H$ and the trivial morphism $0=d_{1,n}^nH$. Here, instead of the initial conditions \eqref{prop:sE:n0},
 we choose $\alpha_k^0$ (for $k=1,\dots,n-1$) to be sections of $\K$ that extend $S_\K$. Then Claim \ref{lemma:sE:claim1} remains true,
 while Claim \ref{lemma:sE:claim3} reads:
\begin{enumerate}[i)]
 \item $d_{n+1,0}^{n+1}\lambda=H$,
 \item $d_{n+1,1}^{n+1}\lambda=:C_\K$ has values in $\K$,
 \item \hspace{6pt}$d_{n,1}^{n+1}\lambda=:C_H$ lifts $\pi(H)$.
\end{enumerate}
Define now $S_B:=\pi(H)$, it is easily seen to be a sphere in $A_B$.
 Moreover $C_H$ bounds $\partial_{n} (S_B)_{x_0}$; in fact $C_H=\tilde{a}$ by construction.
 Then we notice that since $\alpha_{n+1}$ vanishes on $\{t_n=0\}$, the solutions $\tilde{\alpha}_k|_{\{t_n=0\}}$, $k=1,\dots,n-1$ are independent
 of $t_{n+1}$. It makes it easy to show then that $C_\K$ defines a $\K$-homotopy between $S_\K$ the boundary of $C_H$.\end{proof}

\bibliographystyle{habbrv}
\bibliography{sf12}

\end{document}